\documentclass[12pt]{article}

\usepackage{amssymb}
\usepackage{amsthm}
\usepackage{dsfont}
\usepackage{amsmath}
\usepackage{color,graphics,srcltx}
\usepackage{bm}
\usepackage{easybmat}
\usepackage{multirow,bigdelim}
\usepackage{enumitem}
\usepackage{graphicx}
\usepackage{setspace}

\swapnumbers
\newtheorem{proposition}{Proposition}
\newtheorem{example}[proposition]{Example}
\newtheorem{lemma}[proposition]{Lemma}
\newtheorem{corollary}[proposition]{Corollary}
\newtheorem{theorem}[proposition]{Theorem}

\def\prf{\noindent{\it Proof: }}
\def\qed{\hfill$\square$}

\begin{document}

\begin{center}

{
\bf Reflected maxmin copulas and modeling quadrant subindependence }\\[2mm]

Toma\v{z} Ko\v{s}ir, \\
Faculty of Mathematics and Physics, University of Ljubljana, Slovenia\\
Department of Mathematics, Institute of Mathematics, Physics, and Mechanics, Ljubljana, Slovenia\footnote{The authors acknowledge the financial support from the Slovenian Research Agency (research core funding No.\ P1-0222, and research funding No.\ L1-6722). \label{opomba}} \\
E-mail address:$<$tomaz.kosir@fmf.uni-lj.si$>$\\[2mm]

Matja\v{z} Omladi\v{c}, \\
Department of Mathematics, Institute of Mathematics, Physics, and Mechanics, Ljubljana, Slovenia
\textsuperscript{\ref{opomba}}\\[1mm]
E-mail address:$<$matjaz@omladic.net$>$\\[8mm]

{\bf  Abstract}

\end{center}

{\small \noindent

\noindent Copula models have become popular in different applications, including modeling shocks, in view of their ability to describe better the dependence concepts in stochastic systems. The class of maxmin copulas was recently introduced by Omladi\v{c} and Ru\v{z}i\'{c} \cite{OmRu}. It extends the well known classes of Marshall-Olkin and Marshall copulas by allowing the external shocks to have different effects on the two components of the system. By a reflection (flip) in one of the variables we introduce a new class of bivariate copulas called reflected maxmin (RMM) copulas. We explore their properties and show that symmetric RMM copulas relate to general RMM copulas similarly as semilinear copulas relate to Marshall copulas. We transfer that relation also to maxmin copulas. We also characterize possible diagonal functions of symmetric RMM copulas.
\\



 }\\[1mm]

\section{ Introduction }

Dependence concepts play a crucial role in multivariate statistical literature since it was recognized that the independence assumption cannot conveniently describe the behavior of a stochastic system. Since then, different attempts have been made in order to provide more flexible methods to describe the variety of dependence-types that may occur in practice. Copula models have become popular in different applications in view of their ability to describe the relationships among random variables in a flexible way. To this end, several families of copulas have been introduced, motivated by special needs from the scientific practice (cf.\ \cite{DuSe,Joe,Nels}).

Consider, for instance, the case when one wants to build a stochastic model for describing the dependence among
two (or more) lifetimes, i.e. positive random variables. In engineering applications, joint models of lifetimes may serve to estimate the expected lifetime of a system composed by several components. In a related situation like portfolio credit risk, instead, the lifetimes may have the interpretation of time-to-default of firms, or generally financial entities, while a stochastic model may estimate the price/risk of a related derivative contract (e.g.\ CDO). In both cases, it is of interest to estimate the probability of the occurrence of a joint default. The generation of convenient statistical distributions for modeling such situations originated from the seminal paper by Marshall and Olkin \cite{MaOl}, for an up-to-date review see \cite{ChDuMu}. They consider the case of a 2-component system (which may be easily extended to more components) whose behavior is described by the continuous random variables (=r.v.'s) $X_1$ and $X_2$ distributed according to a continuous distribution function $F_i , X_i \sim F_i$ for $i=1,2$. Furthermore for each $i=1,2$ consider the random variable $Z_i$ with probability distribution function $G_i$, that can be interpreted as a shock that effects only the $i$-th component of the system, i.e., the idiosyncratic shock. In addition, consider an r.v.\ $Z_3$ with probability distribution function $G_3$ that can be interpreted as an (exogenous) shock that affects the stochastic behavior of both system components. The life-times of components in this model (which is basically the Marshall model, somewhat more general than the Marshall-Olkin model \cite{MaOl,Mars} -- see also \cite{DuGiMa} for the description of a general framework) are linked with the shocks via $X_1=\min\{Z_1,Z_3\}$ and $X_2=\min\{Z_2,Z_3\}$.

In a recent paper Omladi\v{c} and Ru\v{z}i\'{c} \cite{OmRu} extend results of Marshall and Olkin, cf.\ also \cite{DuOmOrRu}, by introducing a new class of copulas. A simple model where these copulas occur naturally is a two-component system similar to the above in which we assume that the first of the two components has a recovery option. This assumption affects the linkages substantially. As a matter of fact, we get $X_1=\max\{Z_1,Z_3\}$ and $X_2=\min\{Z_2,Z_3\}$. The authors of \cite{OmRu} develop the resulting copula of this model again in a closed form and call it the maxmin copula. Various interpretations of this model can be found in \cite{DuOmOrRu}, since it is possible in many practical situations that the common exogenous shock will produce different effects on different system components. For instance, we may think of $X_1$ and $X_2$ as r.v.'s representing the respective wealth of two groups of people, and the exogenous shock $Z_3$ is interpreted as an event that is beneficial to the first group and detrimental to the second one. Analogously, $X_1$ and $X_2$ can be thought as a short and a long investment, respectively, while $Z_3$ is beneficial only to one of the types of investment.

Maxmin copulas have some properties that are appealing in various contexts related to fuzzy set theory and multicriteria decision making. It includes nonsymmetric copulas that are used for instance as more general fuzzy connectives \cite{BeBoCiSaPlSa,DuKoMeSe}. Its associated measure may have a singular part, a fact of potential use in various copula-based integrals (see \cite{KlLiMePa,KlMeSpSt}). As we shall show in what follows the main idea for maxmin copulas is reflected to probabilistic extensions of semilinear copulas, so it might be worth while to study possible analogues to their extensions to various classes of constructions (cf.\ \cite{JwBaDeMe14,JwBaDeMe15}).

In a seminal paper Durante et al.\ \cite{DuKoMeSe} introduce semilinear copulas and give a surprising relation between  these copulas and Marshall's copulas. The idea of these copulas is roughly that, given a diagonal section $\delta$, they extend linearly across the two triangles. More precisely,
\[
    S_\delta(x,y)=\left\{
             \begin{array}{ll}
               \displaystyle y\frac{\delta(x)}{x}, & \hbox{$y\leqslant x$;} \\
               \displaystyle x\frac{\delta(y)}{y}, & \hbox{otherwise,}
             \end{array}
           \right.
\]
where convention $\frac{0}{0};=0$ is adopted, is called (lower) semilinear copula. The authors of \cite{DuKoMeSe} observe that every symmetric Marshall copula is actually semilinear and moreover \cite[Proposition 8]{DuKoMeSe}, for every Marshall copula $C(x,y)$ there exist semilinear copulas $S_1$ and $S_2$ such that, for all $(x,y)\in[0,1]^2$,
\[
    C(x,y)=\min\left\{\frac{S_1(xy,y)}{y},\frac{S_2(x,xy)}{x}\right\}.
\]

In Section \ref{sec:maxmin} we shortly review the maxmin copulas, introduce the reflected maxmin (RMM) copulas, and give some properties of this class of copulas. Section \ref{sec:sing} is devoted primarily to the singularity and absolute continuity of RMM copulas. In Section \ref{sec:stat} we give a statistical interpretation of this class of copulas, and in Section \ref{sec:diag} we study the properties of the diagonals of symmetric RMM copulas.

\section{ Maxmin copulas revisited }\label{sec:maxmin}

A maxmin copula has two generating functions $\phi,\psi:[0,1]\rightarrow\mathds{R}$ that satisfy the following properties:
\begin{description}
  \item[(F1)] $\phi(0)=\psi(0)=0$, and $\phi(1)=\psi(1)=1$,
  \item[(F2)] $\phi,\psi$ are nondecreasing,
  \item[(F3)] the associated functions $\phi^*:(0,1]\rightarrow[1,\infty)$ and $\psi_*:(0,1)\rightarrow[1,\infty]$ defined by
      \[
        \phi^*(u)=\frac{\phi(u)}{u}\ \ \mbox{and}\ \ \psi_*(v)=\frac{1-\psi(v)}{v-\psi(v)}
      \]
      are nonincreasing. Here, $\psi_*(v)=\infty$ if $\psi(v)=v$.
\end{description}
Given $\phi$ and $\psi$ with the above properties the \emph{maxmin copula} $C(u,v)=C_{\phi,\psi}(u,v)$ is then defined by (cf.\ \cite{OmRu})
\begin{eqnarray}
\nonumber 
  C(u,v) &=& \min\{u,\phi(u)v-\phi(u)\psi(v)+u\psi(v)\} \\
  &=& \min\{u,u+\phi(u)(v-\psi(v))-u(1-\psi(v))\} \\ \nonumber
  &=& u+u(v-\psi(v))\min\{0,\phi^*(u)-\psi_*(v)\},
\end{eqnarray}
where the latter equality holds if $u\neq0$ and $\psi(v)\neq v$.

In order to make the role of the generating functions $\phi$ and $\psi$ more symmetric we apply the following transformation to the class of maxmin copulas. If $C_{\phi,\psi}(u,v)$ is a maxmin copula, then we introduce the copula
\[
    {C}^\sigma_{\phi,\psi}(u,v)=u-C_{\phi,\psi}(u,1-v).
\]
This is a copula of a random vector $(\alpha(U),\beta(V))$ where $\alpha$ is strictly increasing on the range of $U$ and $\beta$ is strictly decreasing on the range of $V$ (cf.\ \cite[Theorem 2.4.4]{Nels}). A general approach to symmetries of copulas is given in \cite[Section 1.7.3]{DuSe}. A simple calculation then gives that
\[
    C^\sigma_{\phi,\psi}(u,v)=\max\{0,uv-(\phi(u)-u)(\widehat{\psi}(v)-v)\},
\]
where $\widehat{\psi}(v)=1-\psi(1-v)$. Here and in the sequel we are using the notation ``$\sigma$'' for the transformation $C\mapsto{C}^\sigma$ only in the case when $C$ is a maxmin copula and the flip is performed on the second variable. Observe that notation $\sigma_2$ is used in this place in \cite{DuSe} thus pointing out that this is a reflection in the second variable, while reflection in the first variable is denoted by $\sigma_1$. We will point out the fact that copula $C^\sigma$ is obtained from the maxmin copula $C$ via the above reflection by terming it \emph{reflected maxmin copula}, or \emph{RMM copula} for short; we will also use abbreviation \emph{SRMM copulas} for the symmetric reflected maxmin copulas. Note that the same reflection sends reflected maxmin copula back to the maxmin copula.

From Properties \textbf{(F1)--(F3)} of the generating functions $\phi$ and $\psi$ it follows that
\begin{align*}
    \phi(0)=\widehat{\psi}(0)=0,\ \ &\phi(1)=\widehat{\psi}(1)=1,\ \ \mbox{and}\ \ \ \ \ \ \ \ \\
    u\leqslant\phi(u),\ \ \ \  &v\leqslant\widehat{\psi}(v) \ \ \ \ \mbox{for all}\ \ \ u,v\in[0,1].
\end{align*}
Also if $\phi(u)=u$ for some $u>0$, then $\phi(u')=u'$ for all $u<u'\leqslant1$. Similarly, if
$\widehat{\psi}(v)=v$ for some $v>0$ then $\widehat{\psi}(v')=v'$ for all $v<v'\leqslant1$. We write
\[
    f(u)=\phi(u)-u,\ \ g(v)=\widehat{\psi}(v)-v=1-v-\psi(1-v),
\]
and
\[
    f^*(u)=\frac{f(u)}{u},\ \ g^*(v)=\frac{g(v)}{v},
\]
for $u,v>0$. In addition we define
\begin{equation}\label{eq:new}
  f^*(0)=\left\{
             \begin{array}{ll}
               \lim_{u\downarrow0}\frac{f(u)}{u}, & \hbox{if it exists;} \\
               \infty, & \hbox{otherwise.}
             \end{array}
           \right.
\end{equation}
and similarly for $g^*(0)$. Hence $f,g:[0,1]\rightarrow[0,1]$ and $f^*,g^*:[0,1]\rightarrow[0,\infty]$.

\begin{lemma}\label{properties G}
Suppose $f,f^*,g$, and $g^*$ correspond to maxmin copula $C_{\phi,\psi}$ as defined above. Then the following holds:
\begin{description}
  \item[(G1)] $f(0)=g(0)=0$, $f(1)=g(1)=0$, $f^*(1)=g^*(1)=0$,
  \item[(G2)] the functions $\widehat{f}(u)=f(u)+u$ and $\widehat{g}(u)=g(u)+u$ are nondecreasing on $[0,1]$.
  \item[(G3)] the functions $f^*$ and $g^*$ are nonincreasing on $(0,1]$.
\end{description}
\end{lemma}

 \prf
In this proof we use the properties of generators $\phi$ and $\psi$ of $C_{\phi,\psi}$ found in \cite{OmRu}. Recall that $f(u)=\phi(u)-u$ and $g(v)=1-v-\psi(1-v)$ to get property \textbf{(G1)} right away. Property \textbf{(G2)} follows by \textbf{(Fc)} and \textbf{(Fd)} of \cite{OmRu}. \textbf{(G3)}: Recall that $\phi^*(u)$ is nonincreasing, so that $f^*(u)= \phi^*(u)-1$ is also nonincreasing. Finally, the function
\[
    \widetilde{\psi}(v)=\frac{1-\psi(v)}{1-v}
\]
is nondecreasing on $(0,1]$ by \cite[Lemma 3]{OmRu}. So, $g^*(v)=\widetilde{\psi}(1-v)-1$ is nonincreasing on $(0,1]$.
 \qed\\

As noted in \cite[p.\ 117]{OmRu} the functions  $\phi$ and $\psi$ are continuous on $(0,1)$, the function $\phi$ is not necessarily continuous at 0, and the function $\psi$ is not necessarily continuous at 1.

\textbf{Comment.}  Observe that Condition \textbf{(G3)} allows for a very simple geometric interpretation: The angle between the vectors $(u, f(u))$ and $(u, 0)$ is nonincreasing. \\

\textbf{Definition.} Given the functions $f,g:[0,1]\rightarrow[0,\infty)$ which satisfy Properties \textbf{(G1)--(G3)} of Lemma \ref{properties G}, we define a reflected maxmin copula $C_{f,g}$ by the following rule
\begin{equation}\label{inverse maxmin}
    C_{f,g}(u,v)=\max\{0,uv-f(u)g(v)\}.
\end{equation}
Note that we can rewrite $C_{f,g}$ as follows
\begin{equation}\label{inverse maxmin star}
    C_{f,g}(u,v)=uv\max\{0,1-f^*(u)g^*(v)\}.
\end{equation}

It follows by Lemma \ref{properties G} that we can associate to every maxmin copula a corresponding RMM copula. The next lemma shows the converse of this fact.

\begin{lemma}\label{properties F}
Suppose the functions $f,f^*,g$, and $g^*$ satisfy properties \mbox{\rm \textbf{(G1)--(G3)}} of Lemma \ref{properties G}. Then the functions
\[
    \phi(u)=u+f(u),\ \ \mbox{and}\ \  \psi(v)=v-g(1-v),
\]
satisfy properties \mbox{\rm\textbf{(F1)--(F3)}} and $C_{\phi,\psi}$ is a maxmin copula.
\end{lemma}

\begin{proof}
Property \textbf{(F1)} is straightforward. Now, since $\phi(u)= \widehat{f}(u)$ for all $u\in[0,1]$, property \textbf{(G2)} implies that $\phi$ is nondecreasing. Similarly, $\widehat{g}$ is nondecreasing by \textbf{(G2)}, so that
\[
    v_1+g(v_1)\leqslant v_2+g(v_2),\quad\mbox{and therefore}\quad 1-v_2+g(1-v_2)\leqslant 1-v_1+g(1-v_1),
\]
for any $v_1,v_2$ such that $0\leqslant v_1\leqslant v_2\leqslant1$. So, $\psi(v_2)-\psi(v_1)=1-v_1+g(1-v_1)-(1-v_2+g(1-v_2))\geqslant0$. Hence, $\psi$ is nondecreasing and \textbf{(F2)} is proved as well.

To show \textbf{(F3)} recall first that $\phi^*(u)=f^*(u)+1$ for all $u\in(0,1]$ so that function $\phi^*$ is nonincreasing since function $f^*$ is nonincreasing by \textbf{(G3)}. Next, consider
\begin{equation*}
  \psi_*(v)-1=\frac{1-\psi(v)-v+\psi(v)}{v-\psi(v)}=\frac{1}{g^*(1-v)},
\end{equation*}
and recall that the function $g^*(v)$ is nonincreasing by property \textbf{(G3)} implying that $g^*(1-v)$ is nondecreasing and therefore $\displaystyle \frac{1}{g^*(1-v)}$ is nonincreasing. So, $\psi_*$ is nonincreasing on $(0,1)$ as desired and \textbf{(F3)} holds. (Observe that here, one may want to consider the possibility of existence of a $u<1$ such that $g(u)=0$ as a special case.)
\end{proof}

Using this fact we show in the next theorem immediately that $C_{f,g}$ defined by \eqref{inverse maxmin} or equivalently \eqref{inverse maxmin star} is a copula. But, before we do that, let us pause for a few comments. First, the copulas defined by these two formulas remind us of a class of copulas introduced in \cite{RoLaUbFl}. Copulas of this type also appeared in \cite[Proposition 3.2]{DuJa}. As a matter of fact the authors of \cite{RoLaUbFl} study there copulas of the form $C(u,v)=uv+f(u)g(v)$. Since we allow that $uv+f(u)g(v)<0$ for some $u,v\in[0,1]$, our copulas $C_{f,g}$ extend a subclass of those. Moreover, our copulas may be viewed as perturbations of the product copula $\Pi$.  General perturbations of copulas were studied in \cite{DuFeSaUbFl} and \cite{MeKoKo}, where a subclass of what we call reflected maxmin copulas were considered (cf.\ \cite[\S 3]{MeKoKo}). Formula \eqref{inverse maxmin} can also be related to the formula in \cite[Theorem 7.1]{DuMePaSe} with maximum replaced by minimum; we believe that similarity is more than just coincidental. An extension of our copulas in aggregation context may bring us to more general generating functions on one hand and possibly to more general class of bivariate functions (that are currently copulas), i.e.\ to quasi-copulas.

\begin{theorem}\label{copula}
Let $f$ and $g$ satisfy {\rm \textbf{(G1)--(G3)}}, then $C_{f,g}(u,v)$ is a copula. For $\phi(u)=u+f(u)$ and $\psi(v)=v-g(1-v)$ we have
\[    C_{\phi,\psi}(u,v)=u-C_{f,g}(u,1-v) \]
and
\[    C_{f,g}(u,v)=u-C_{\phi,\psi}(u,1-v). \]
\end{theorem}

\begin{proof}
  By Lemma \ref{properties F} the functions $\phi$ and $\psi$ satisfy conditions \textbf{(F1)}--\textbf{(F3)} so that $C_{\phi,\psi}$ is a maxmin copula by definition. So
\[    C_{f,g}(u,v)=u-C_{\phi,\psi}(u,1-v), \]
is a copula by \cite[Thm.\ 2]{Nels}.
\end{proof}

Observe that the product copula $\Pi(u,v)$ is an RMM copula for $f\equiv0,g\equiv0$. Also, the Fr\'{e}chet-Hoeffding lower bound $W(u,v)=\max\{0,u+v-1\}$ is an RMM copula if we take
\[
    f(t)=g(t)=\left\{
                \begin{array}{ll}
                  1-t, & \hbox{if $0<t\leqslant1$;} \\
                  0, & \hbox{$t=0$.}
                \end{array}
              \right.
\]
Actually, these are the lower and upper bound for the whole class of RMM copulas.

\begin{lemma}\label{lemma:subindependent}
If $C(u,v)$ is an RMM copula then
\[
    W(u,v)\leqslant C(u,v)\leqslant \Pi(u,v)
\]
\end{lemma}

 \prf
This is a straightforward consequence of the definition since $C(u,v)=C_{f,g}(u,v)=\max\{0,uv-f(u)g(v)\}$ for some $f(u),g(v)\geqslant0$.
 \qed\\

By the properties of generating functions of maxmin copulas it follows that $f(u)$ and $g(v)$ are continuous on $(0,1]$, but they need not be continuous at $0$.

If we take $f=g$ then the RMM copula is symmetric, so that it is an SRMM copula. We write
\[
    C_f(u,v)=C_{f,f}(u,v)=\max\{0,uv-f(u)f(v)\}=uv\max\{0,1-f^*(u)f^*(v)\}.
\]
Note that by Lemma \ref{lemma:subindependent} every RMM copula is negatively quadrant dependent (cf.\ \cite{Nels}). This fact might encourage us to call the kind of copulas equivalently to be quadrant subindependent.

The diagonal section of an RMM copula is of the form
\[
    \delta(t)=C_{f,g}(t,t)=\max\{0,t^2-f(t)g(t)\}.
\]
Since $f(t)g(t)\geqslant0$, it follows that $0\leqslant \delta(t)\leqslant t^2$ for all $t\in[0,1]$.

\begin{proposition}
If a copula is simultaneously an RMM copula and a semilinear copula, then it is equal to the product copula.
\end{proposition}

 \prf
This is immediate by the fact that a semilinear copula is PQD by \cite[Corollary 5]{DuKoMeSe} and an RMM copula is NQD by Lemma \ref{lemma:subindependent}.
 \qed\\

So, RMM copulas seem to have a complementary position to the semilinear copulas. Note also that RMM copulas cannot be obtained from semilinear copulas via reflections (also called flipping transformations).

RMM copulas may have additional asymmetries. The symmetry (or exchangeability) property may be too restrictive, as discussed in \cite{GeNe}. Here we have pointed out some properties of the symmetric RMM copulas that are not shared by the whole class of RMM copulas. Some more will be given in Section \ref{sec:diag}.

\section{ Some specific properties of RMM copulas }\label{sec:sing}

In this section we will consider primarily the singularity and absolute continuity of RMM copulas. Here is some notation to start with. For an RMM copula $C=C_{f,g}$ we introduce the \emph{stand of} $C$ (just for the sake of this paper). It is denoted by $\mathcal{S}(C)$ and defined as the closure of the set
\[
    \{(u,v)\in[0,1]^2\,;\, f(u)g(v)< uv\}= \{(u,v)\in[0,1]^2\,;\, f^*(u)g^*(v)< 1\}.
\]
It is also the closure of the set
\[
    \{(u,v)\in[0,1]^2\,;\, C(u,v)>0\}.
\]
The closure of the complement of $\mathcal{S}(C)$ will be called the \emph{zero set of} $C$ and denoted by $\mathcal{Z}(C)$; clearly

\[
    \mathcal{Z}(C)=\{(u,v)\in[0,1]^2\,;\, C(u,v)=0\}.
\]

Here are some examples of RMM copulas with or without singular component in addition to $W(u,v)$ for which we already know that it is fully singular RMM copula.

\begin{example}\label{regular}
  There are families of RMM copulas that are absolutely continuous.
\end{example}

\begin{proof}
  The functions
\[{\renewcommand{\arraystretch}{1.5}
    f_a(t)=\left\{
             \begin{array}{ll}
               at, & \hbox{$\displaystyle 0\leqslant t\leqslant \frac{1}{2}$,} \\
               a(1-t), & \hbox{$\displaystyle\frac{1}{2}\leqslant t\leqslant 1$,}
             \end{array}
           \right.\quad \hbox{and}\quad g_b(t)=bt(1-t),
}\]
for $a,b\in(0,1)$ satisfy conditions {\textbf{(G1)}}--{\textbf{(G3)}}. The corresponding two-parameter family of RMM copulas $C:=C_{f_a,g_b}(u,v)$ has the densities
\[
    c=\frac{\partial^2 C}{\partial u\partial v}\quad \mbox{such that}\quad \int_{0}^{1}\int_{0}^{1} c\, du\, dv=1
\]
as a straightforward computation reveals. The desired conclusion then follows by \cite{DuSe,Nels}. Note that $c$ is nonzero everywhere on $[0,1]^2$. Observe also that if we replace $a,b$ by some $a',b'\in(0,1)$ such that $ab=a'b'$, the copula $C$ remains unchanged. So, the generators of an RMM copula $C$ are not uniquely determined by it (cf. the remark at the end of this section).
\end{proof}

\begin{example}
  There is a class of SRMM copulas that belongs to the class of EFGM (Eyraud-Farlie-Gumbel-Morgenstern\footnote{Here we are following the notiation of \cite[Section 6.3]{DuSe}}) copulas.
\end{example}

\begin{proof}
  Let $C_a=C_{f_a}$ be a class of SRMM copulas generated by $f_a(t)=at(1-t)$ for $a\in(0,1)$. Then
\[
    C_a(u,v)=uv-a^2uv(1-u)(1-v)
\]
are EFGM copulas (see \cite[p.\ 193]{DuSe}). They are absolutely continuous with density $c_a(u,v)=1-a^2+2a^2(u+v-2uv)$.
\end{proof}

\begin{example}\label{example:singular}
  There are families of RMM copulas that have both a nontrivial absolutely continuous component and a singular component.
\end{example}

\begin{proof}
  We give three examples of the kind with slightly different properties.

\textbf{(a)} Functions
\[{\renewcommand{\arraystretch}{1.5}
    f_\theta(t)=\left\{
             \begin{array}{ll}
               \displaystyle\left(\frac{1}{\theta}-1\right) t, & \hbox{$0\leqslant t\leqslant \theta$,} \\
               (1-t), & \hbox{$\theta\leqslant t\leqslant 1$,}
             \end{array}
           \right.\quad \hbox{and}\quad g_\eta(t)=\left\{
             \begin{array}{ll}
               \displaystyle\left(\frac{1}{\eta}-1\right) t, & \hbox{$0\leqslant t\leqslant \eta$,} \\
               (1-t), & \hbox{$\eta\leqslant t\leqslant 1$,}
             \end{array}
           \right.
}\]
for $\theta,\eta\in(0,1)$, 
satisfy conditions {\textbf{(G1)}}--{\textbf{(G3)}}. Each copula of the corresponding two-parameter family of RMM copulas $C:=C_{f_\theta, g_\eta}(u,v)$ has the density
\[
    c=\frac{\partial^2 C}{\partial u\partial v}\quad \mbox{such that}\quad \int_{0}^{1}\int_{0}^{1} c\, du\, dv=\min\{\theta+\eta,1\}
\]
as seen via a straightforward computation. So, copula $C$ under consideration has both the continuous component and the singular component if $\theta+\eta<1$. The latter is distributed uniformly along the segment
\begin{equation}\label{singular_a}
  \{(u,1-u)\,;\,u\in[\theta,1-\eta]\},
\end{equation}
as it is immediately checked. In the limit $\theta,\eta\rightarrow0$ we obtain copula $W$.

\textbf{(b)} A one-parameter family of SRMM copulas is given by the function
\[{\renewcommand{\arraystretch}{1.5}
    f_\delta(t)=\left\{
                  \begin{array}{ll}
                    0, & \hbox{$t=0$;} \\
                    \delta, & \hbox{$0<t\leqslant \delta$;} \\
                    t, & \hbox{$\displaystyle\delta\leqslant t\leqslant \frac{1}{2}$;} \\
                    1-t, & \hbox{$\displaystyle\frac{1}{2}\leqslant t\leqslant1$,}
                  \end{array}
                \right.
}\]
for $\delta\in\left(0,\frac{1}{2}\right)$, satisfying conditions {\textbf{(G1)}}--{\textbf{(G3)}}. After performing a direct computation we obtain that the integral of the density function of these copulas over $[0,1]^2$ equals to $1-2\delta \ln2$. The singular component is distributed along the two arcs
\begin{equation}\label{singular_b}
  \left\{\left(u,\frac{\delta(1-u)}{u}\right)\,;\,
u\in\left[\frac{1}{2},1\right]\right\}\quad\mbox{and}\quad
\left\{\left(\frac{\delta(1-v)}{v},v\right)\,;\,
v\in\left[\frac{1}{2},1\right]\right\}
\end{equation}
and we can verify quickly that the mass of each arc equals $\delta \ln2\approx 0.69\,\delta$.

\textbf{(c)} The functions
\[
    f(t)=t(1-t)\quad\mbox{and}\quad g_\mu(t)=\left\{
                                               \begin{array}{ll}
                                                 0, & \hbox{$t=0,t\geqslant \mu$;} \\
                                                 \mu-t, & \hbox{$0<t\leqslant \mu$,}
                                               \end{array}
                                             \right.
\]
for $\mu\in(0,1]$ satisfy conditions {\textbf{(G1)}}--{\textbf{(G3)}}. Each copula of the corresponding one-parameter family of RMM copulas has the density whose integral over the unit square is equal to $1-\mu\,(\ln4-1)\approx1-0.38\, \mu$. A painless computation shows that the singular mass is distributed along the arc
\begin{equation}\label{singular_c}
  \left\{\left(u,\frac{\mu(1-u)}{2-u}\right)\,;\,u\in[0,1]\right\}.
\end{equation}
\end{proof}

To be able to show the differences of Examples \ref{example:singular} let us first recall the notation for the level set of a copula $C$
\[
    L_t(C)=\{(u,v)\in[0,1]^2\,;\,C(u,v)=t\}
\]
for any $t\in(0,1]$. For $t=0$ the level set is defined differently (see \cite{Nels,DuSe}), i.e.
\[
    L_0(C)=\mathcal{S}(C)\cap\mathcal{Z}(C)\cap(0,1)^2,
\]
if this intersection is nonempty, or
\[
    L_0(C)=\{[t,0],[0,t]\,;\,t\in[0,1]\}
\]
if $\mathcal{S}(C)=[0,1]^2$.

So, in the case of RMM copulas all of the level sets are curves given by equations
\[
    uv-f(u)g(v)=t.
\]
It is not always possible to give an explicit equation of either the form $u=\varphi_t(v)$ or $v=\psi_t(v)$ for set $L_t$ as it is seen for $L_0$ of the family of copulas presented in Example \ref{example:singular}\textbf{(b)}. There $L_0$ is the union of arcs given by \eqref{singular_b} and segments
\[
    \left\{\left(u,\frac{1}{2}\right)\,;\,\delta\leqslant u\leqslant\frac{1}{2}\right\}\cup \left\{\left(\frac{1}{2},v\right)\,;\,\delta\leqslant v\leqslant\frac{1}{2}\right\}.
\]
 Also, level curves $L_t$ for small $t>0$ of the same example are not convex functions.

Now, among copulas of Example \ref{example:singular} with nontrivial singular component there are those, whose singular component is only a part of the boundary $L_0(C)$ of their stand. This happens, say, in Example \ref{example:singular}\textbf{(a)} with $\theta+\eta<1$, where the singular component is given by Formula \eqref{singular_a}; and also in Example \ref{example:singular}\textbf{(b)}. On the other hand, the copulas presented in Example \ref{example:singular}\textbf{(c)} have their singular components distributed along the entire level set $L_0(t)$ expressed by \eqref{singular_c}; and the same is true for copula $W$. Among the absolutely continuous copulas there are those with density nonzero everywhere on the unit square as is the case for copulas of Example \ref{regular}, and those whose density have parts where it is identically equal to zero as in Example \ref{example:singular}\textbf{(a)} with $\theta+\eta=1$.

\begin{figure}[h!]
  \centering
\includegraphics[width=0.32\textwidth]{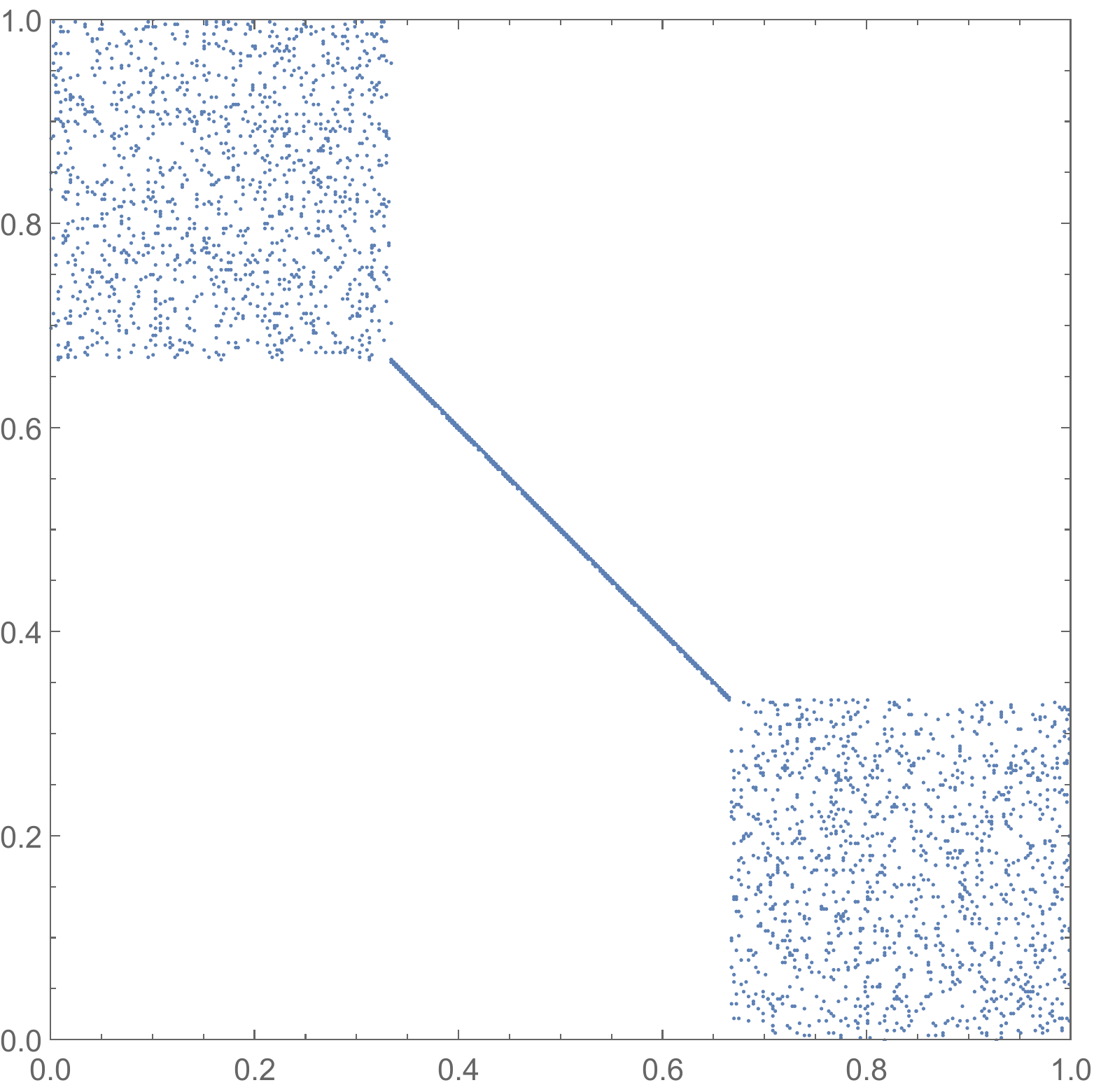} \hfil \includegraphics[width=0.32\textwidth]{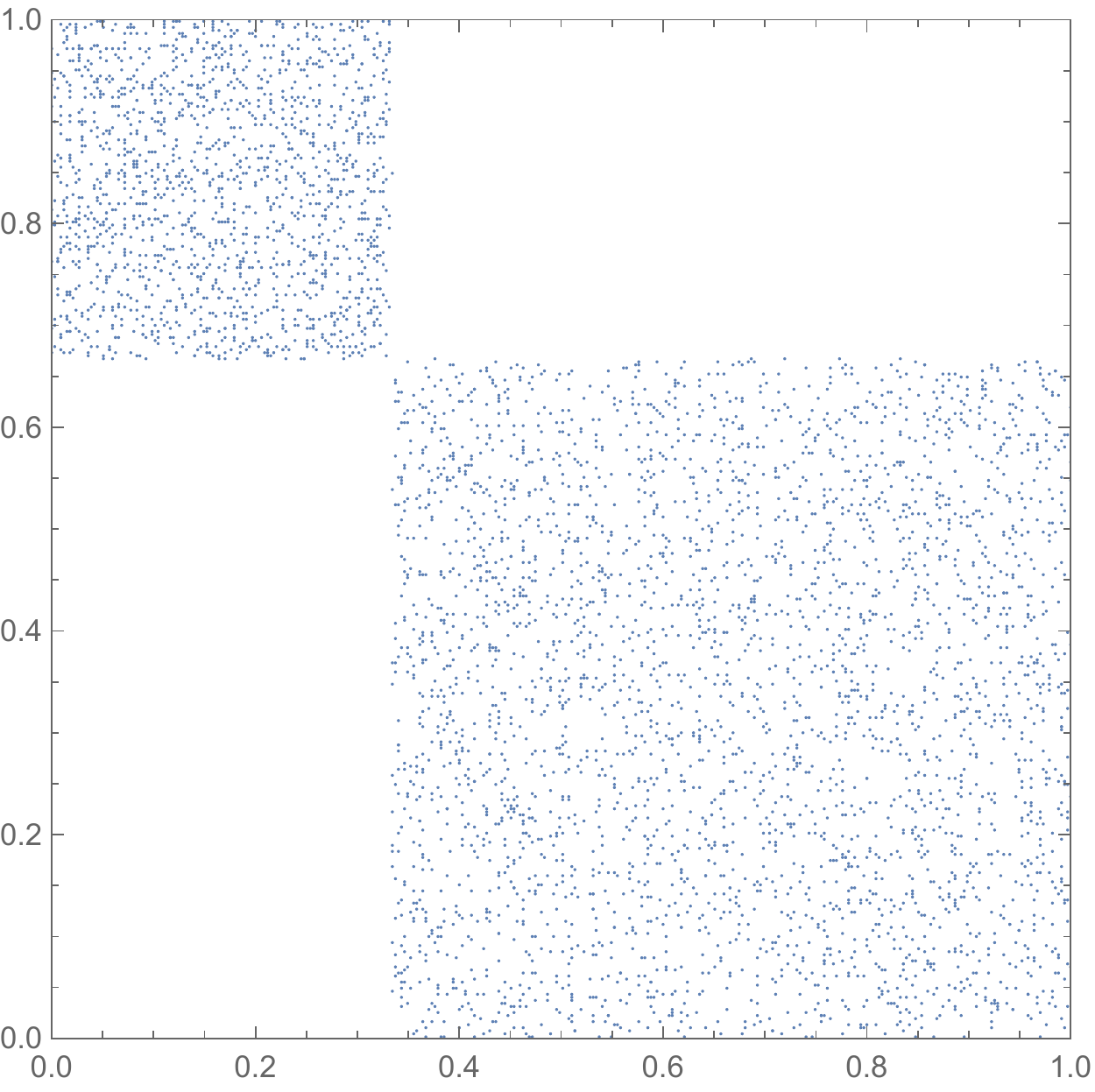} \hfil \includegraphics[width=0.32\textwidth]{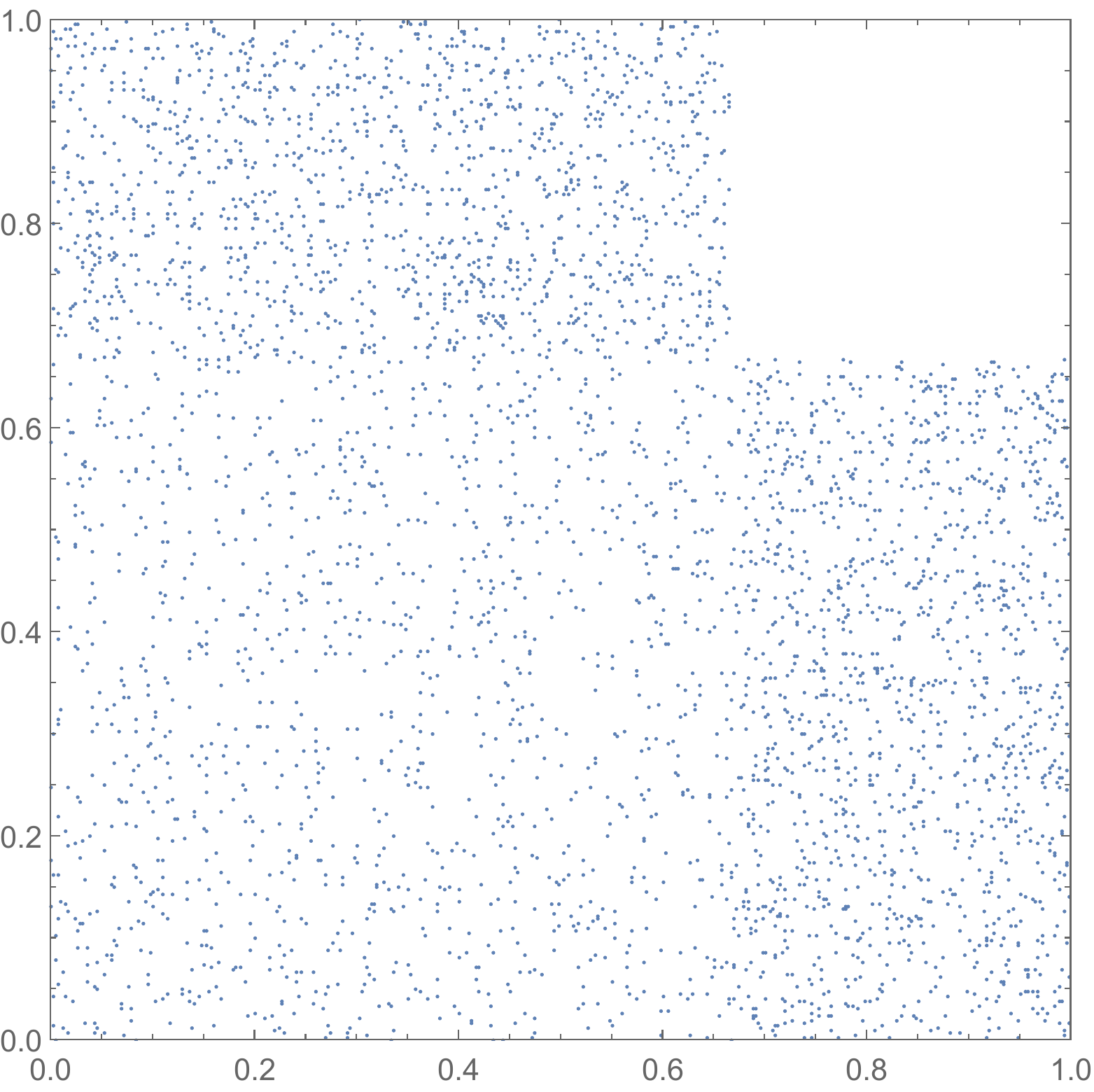} \\ \includegraphics[width=0.32\textwidth]{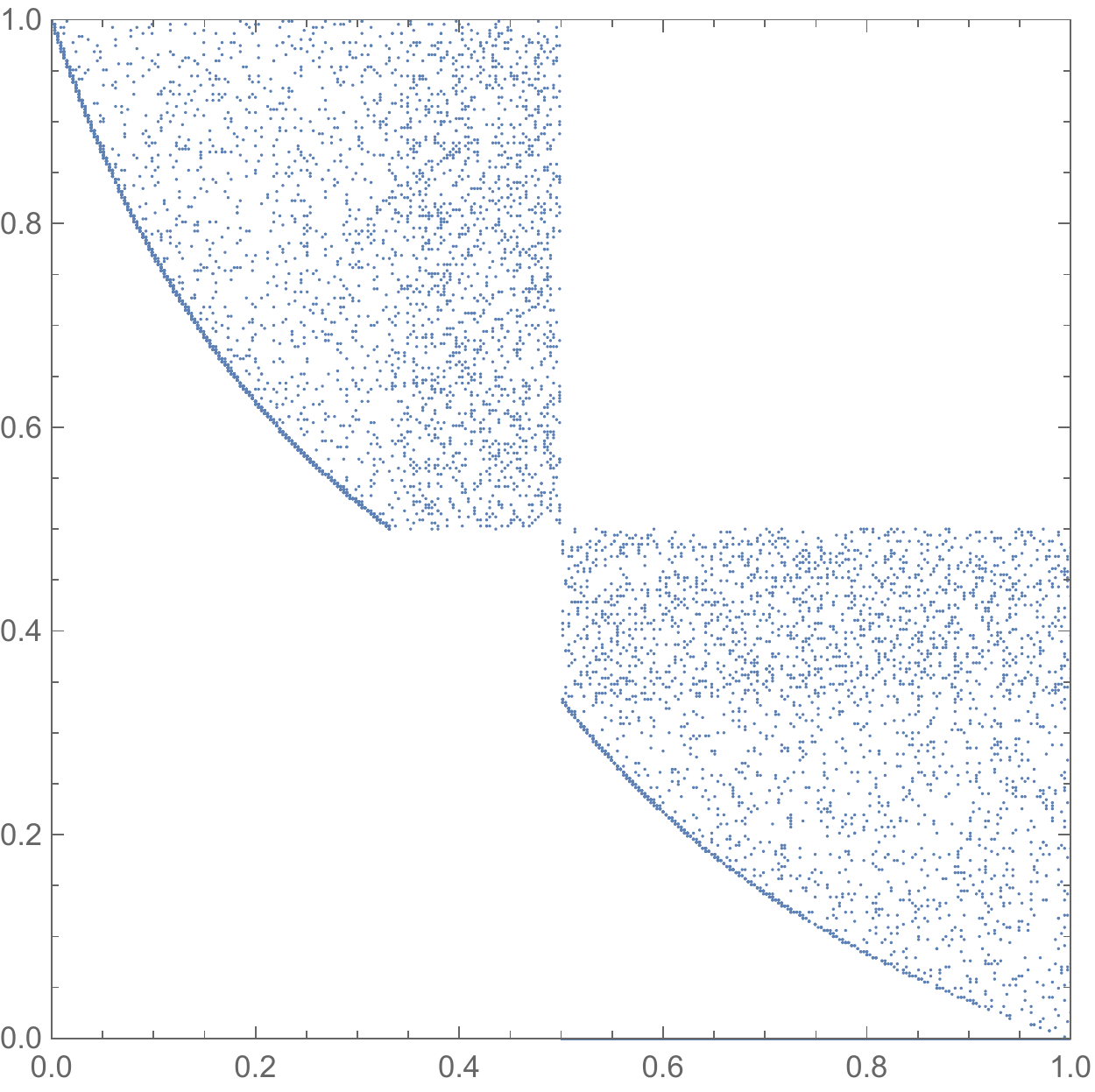} \hfil \includegraphics[width=0.32\textwidth]{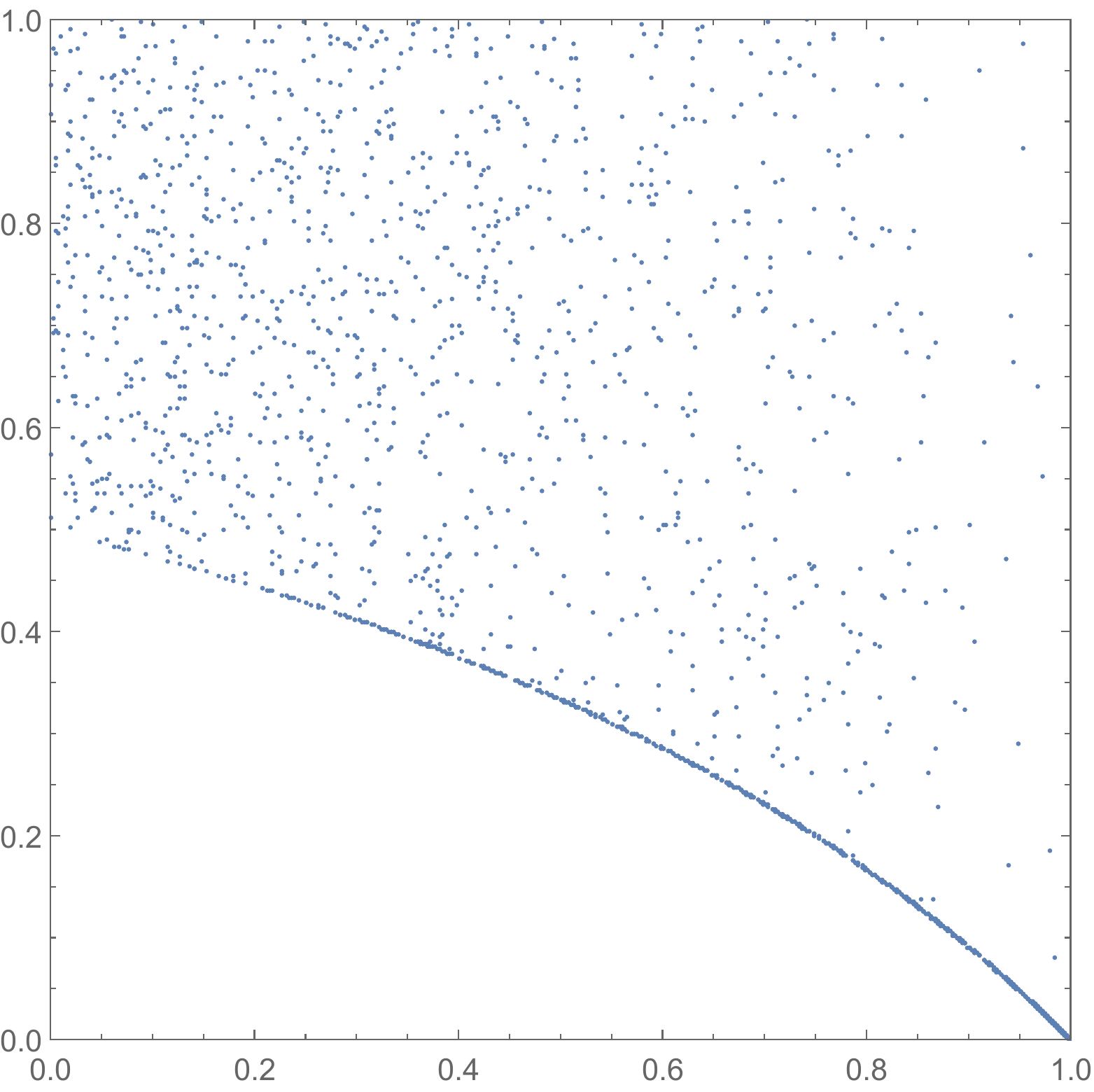} \hfil \includegraphics[width=0.32\textwidth]{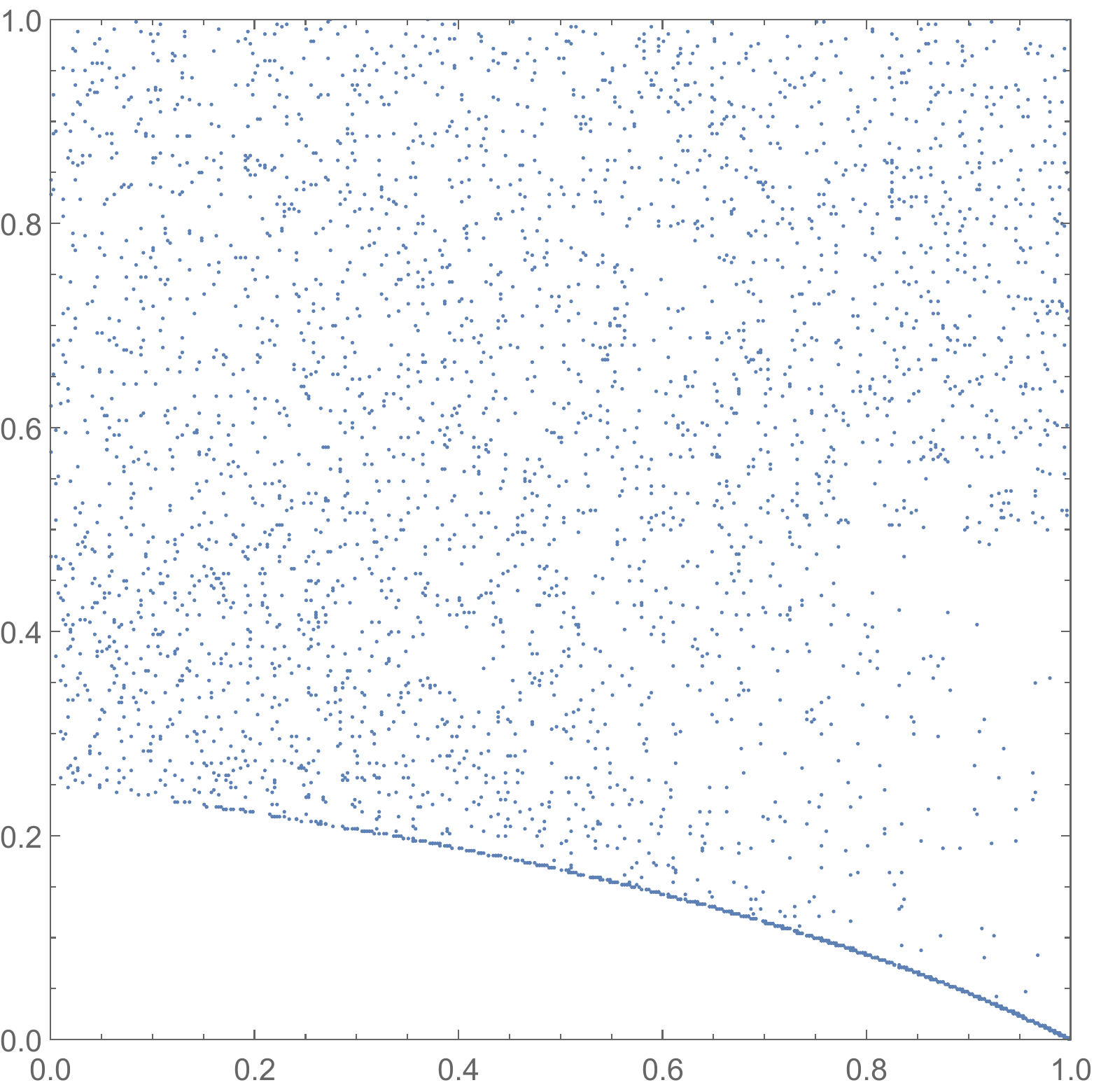}
  \caption{Copulas of Example \ref{example:singular}}
\end{figure}

\begin{spacing}{1.1}
Scatterplots of some of the copulas from Example \ref{example:singular} are presented in Figure 1. In the first row of the figure, there are three copulas from \ref{example:singular}\textbf{(a)}, i.e.\ those for $\displaystyle \theta=\eta=\frac13$, $\displaystyle\theta=\frac13$, $\displaystyle \eta=\frac23$, and $\displaystyle\theta=\eta=\frac23$, respectively. The first one has a nontrivial singular part, while the other two are absolutely continuous. In the second row, there are copulas from \ref{example:singular}\textbf{(b)} for $\displaystyle\delta=\frac13$, and \ref{example:singular}\textbf{(c)} for $\displaystyle\mu=1$ and $\displaystyle\mu=\frac12$, respectively.
\end{spacing}

\begin{figure}[h!]  \centering  \includegraphics[width=0.32\textwidth]{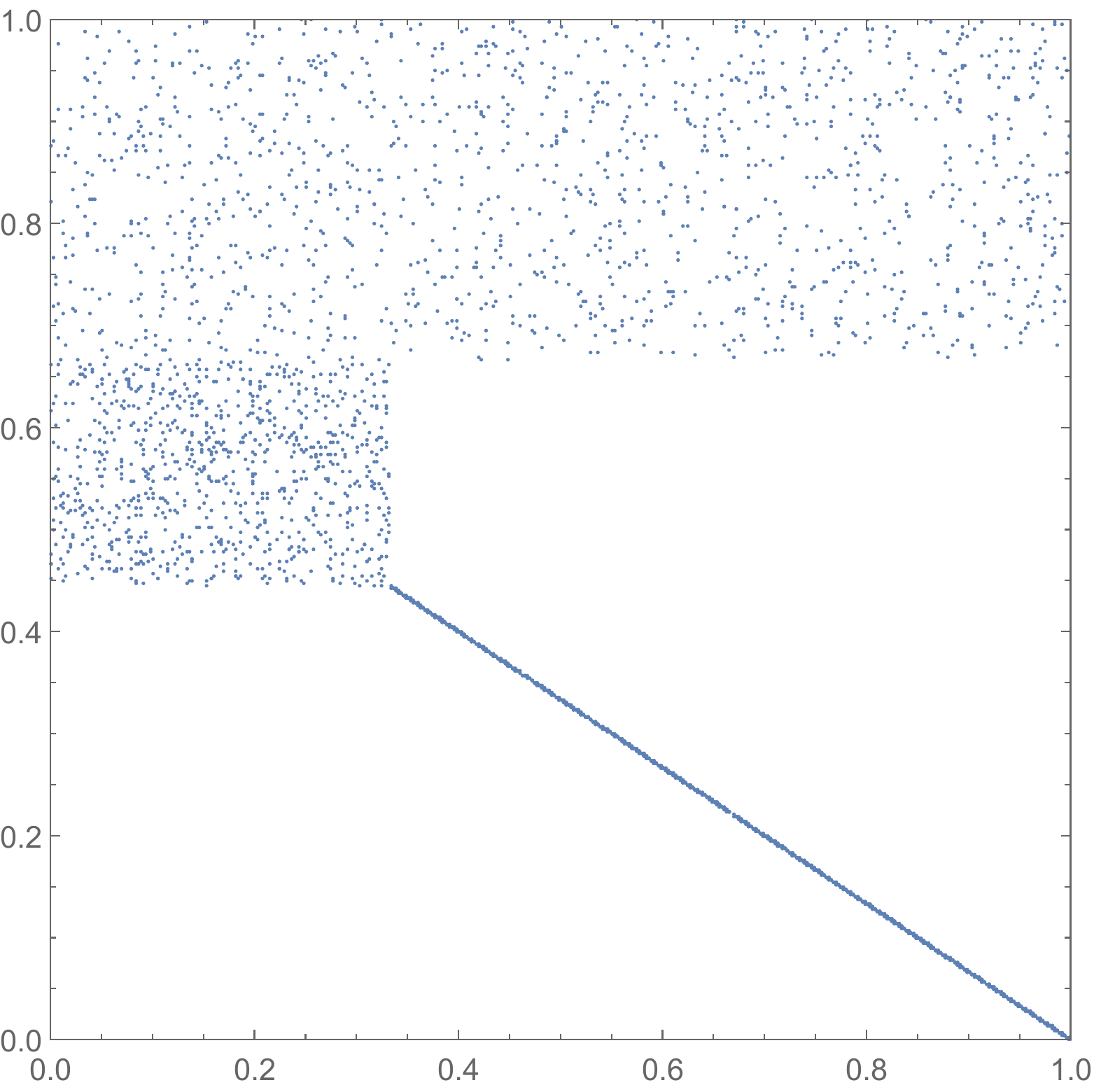} \hfil \includegraphics[width=0.32\textwidth]{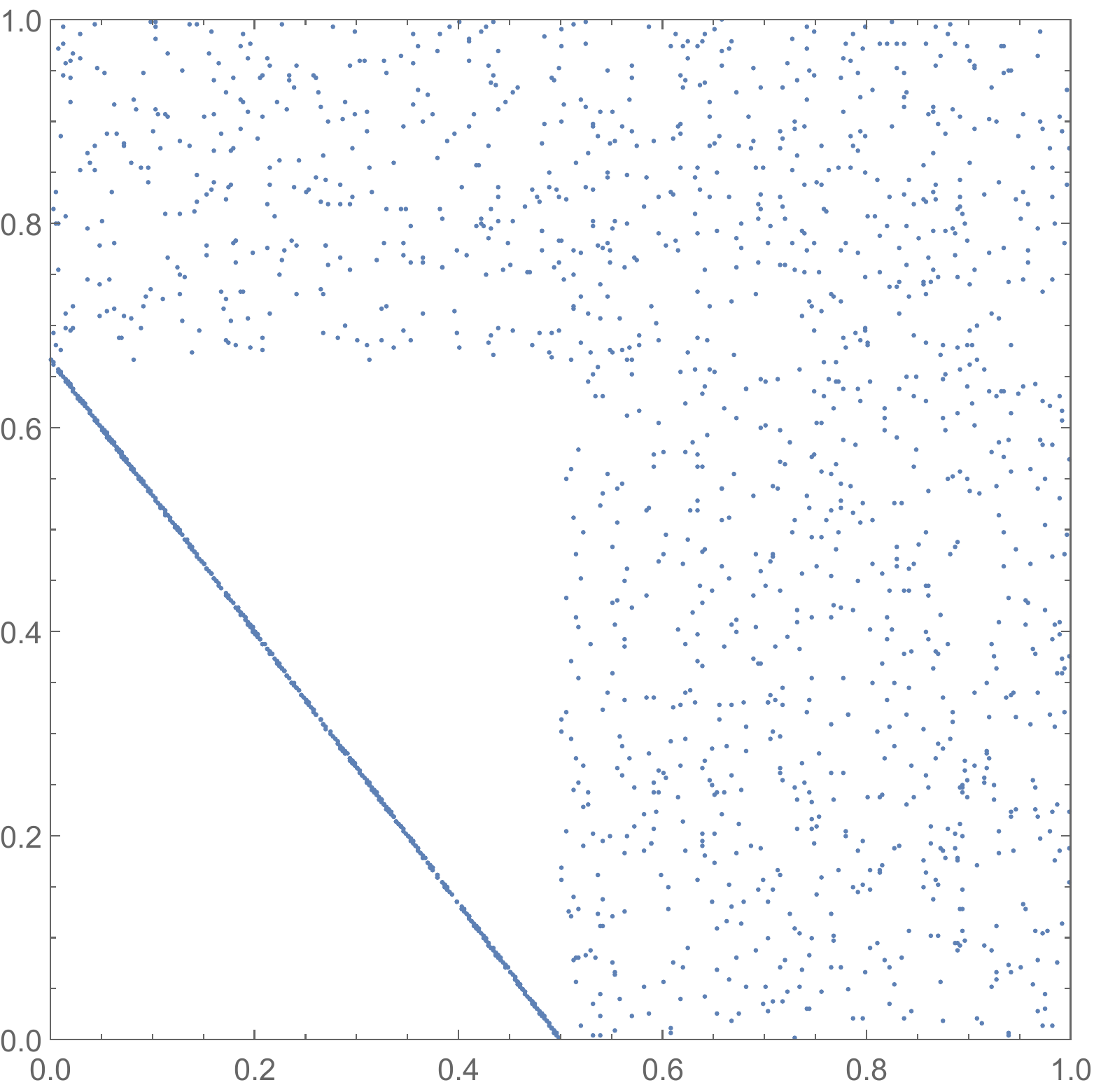} \hfil \includegraphics[width=0.32\textwidth]{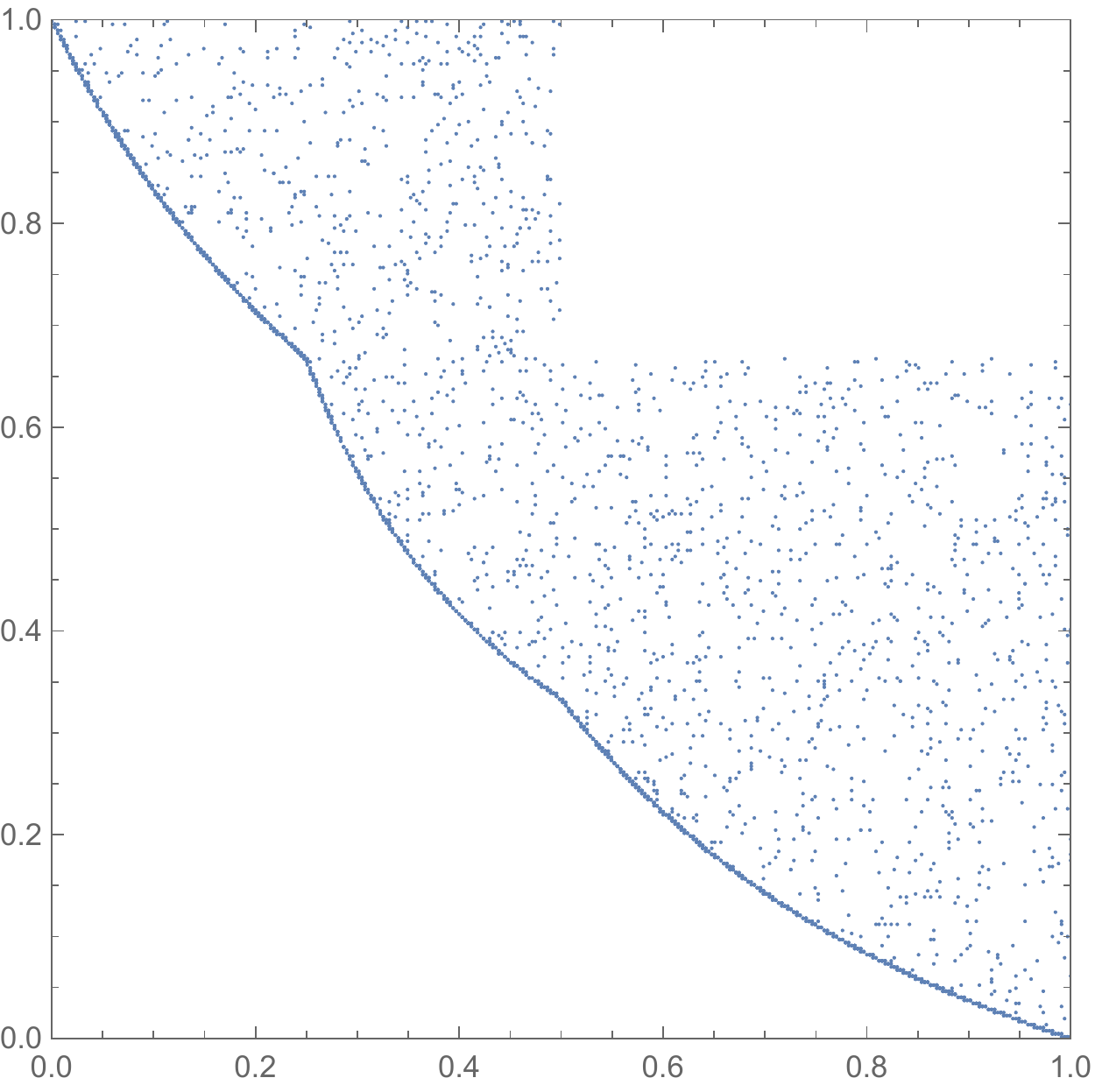} \\ \includegraphics[width=0.32\textwidth]{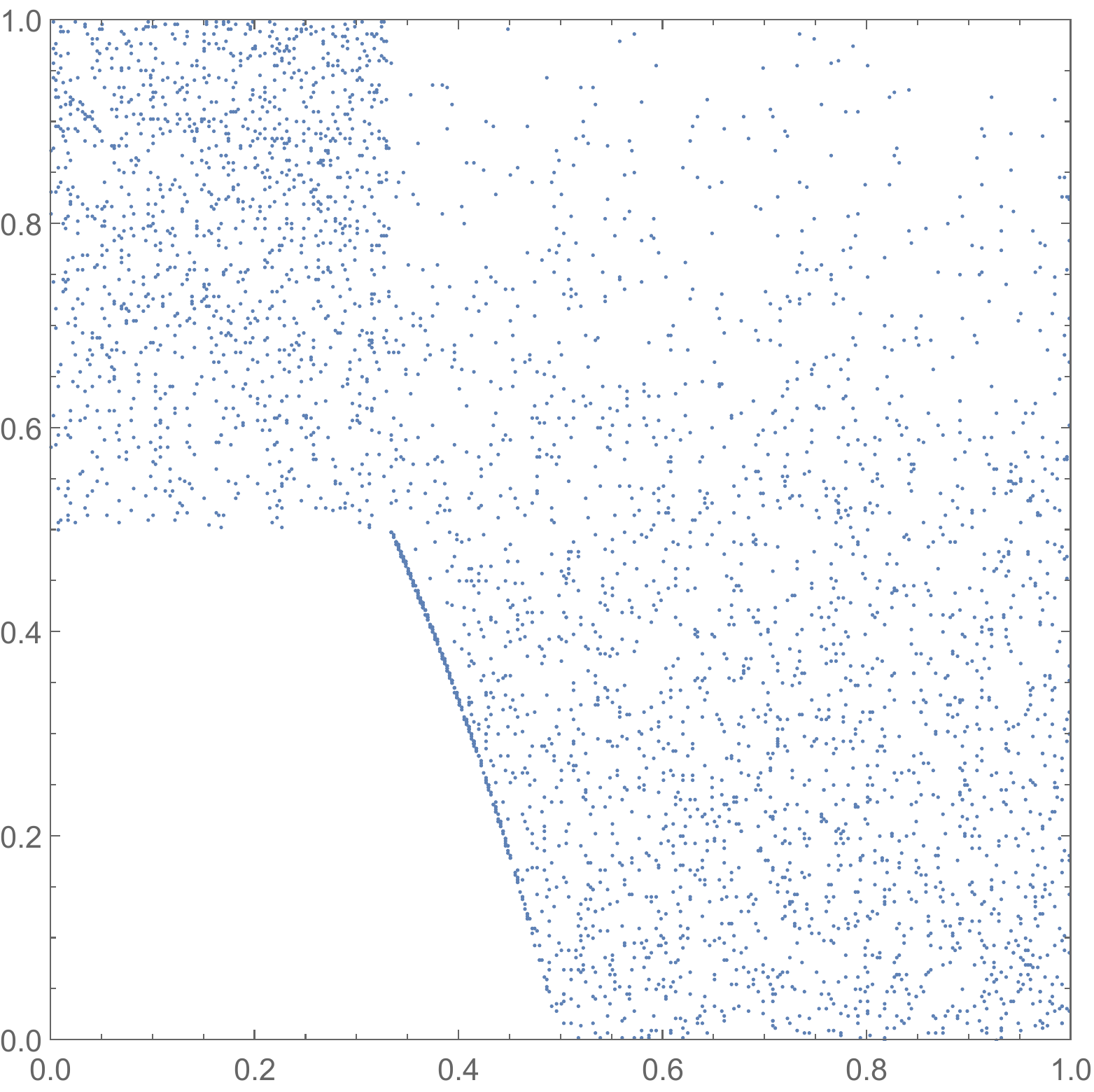} \hfil \includegraphics[width=0.32\textwidth]{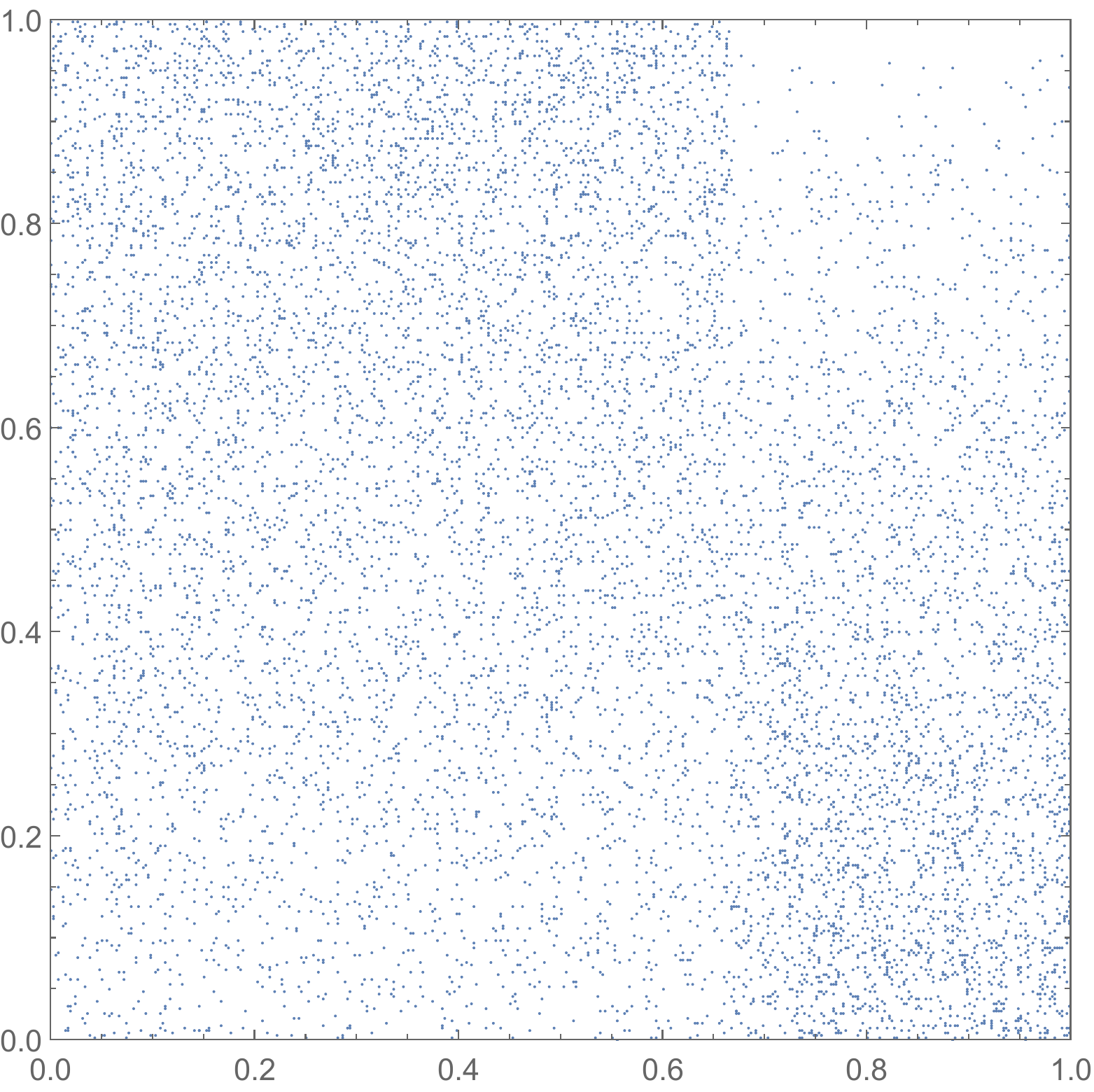}   \hfil \includegraphics[width=0.32\textwidth]{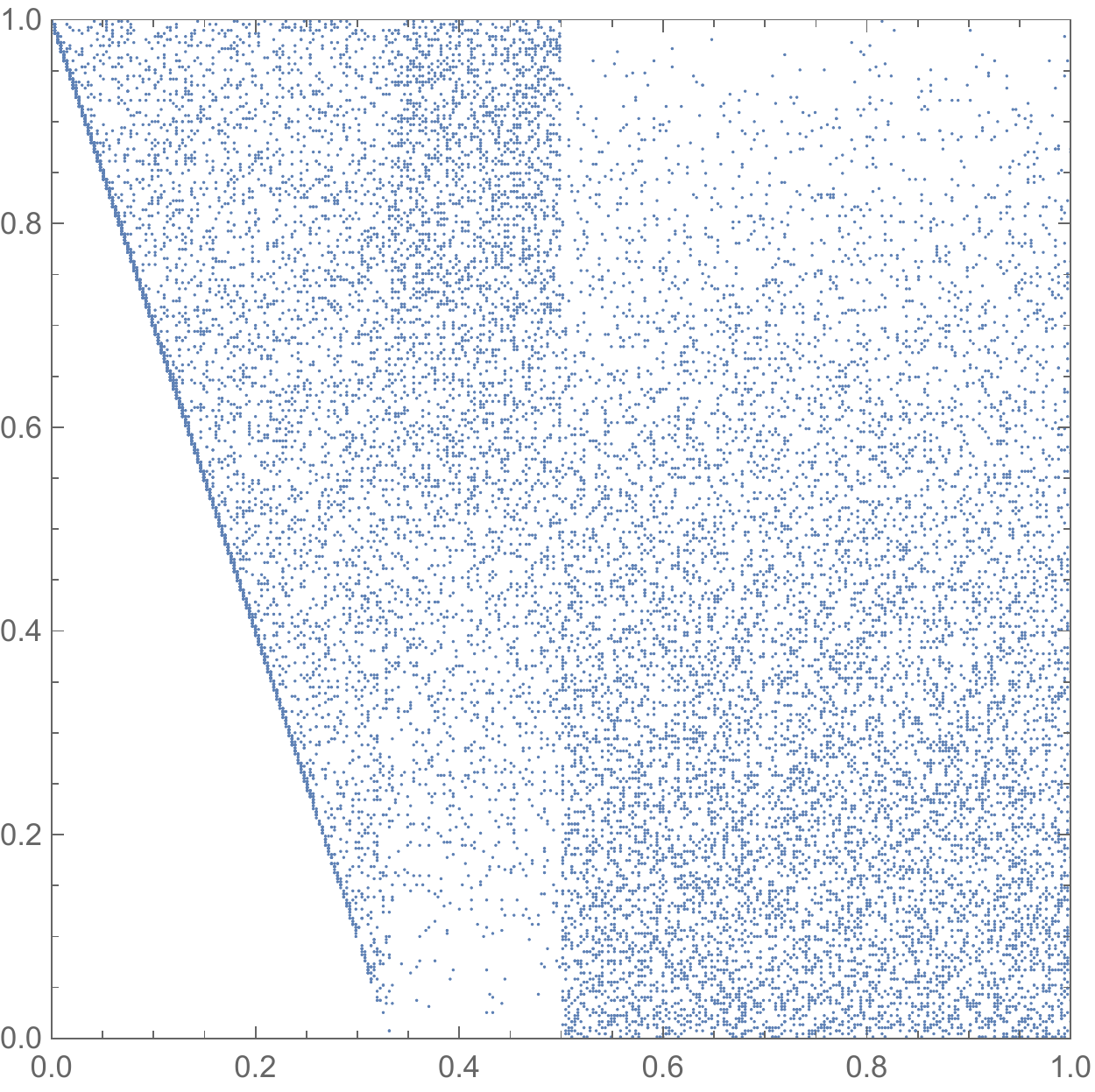}            \caption{ Some more reflected maxmin copulas}   \end{figure}

We include some more scatterplots of RMM copulas in Figure 2. The generators of the copulas in the first row are respectively equal to
\[
{\renewcommand{\arraystretch}{1.5}
    f(t) = \left\{
             \begin{array}{ll}
               2t, & \hbox{if $\displaystyle0\leqslant t\leqslant \frac13$;} \\
                1 - t, & \hbox{otherwise;}
             \end{array}
           \right.
\ \ \mbox{and}\ \ g(t)= \left\{
                          \begin{array}{ll}
                            \displaystyle \frac23 - t, & \hbox{if $\displaystyle 0< t\leqslant \frac23$;} \\
                            0, & \hbox{otherwise;}
                          \end{array}
                        \right.}
\]
\[
{\renewcommand{\arraystretch}{1.5}
    f(t)=\left\{
           \begin{array}{ll}
             \displaystyle \frac12 - t, & \hbox{if $\displaystyle  0< t\leqslant \frac12$; } \\
             0, & \hbox{otherwise;}
           \end{array}
         \right.
    \ \ \mbox{and}
    \ \ g(t)=\left\{
                 \begin{array}{ll}
                   \displaystyle\frac23 - t, & \hbox{if $\displaystyle 0< t\leqslant \frac12$; } \\
                   0, & \hbox{otherwise;}
                 \end{array}
               \right.
}\]
and
\[
{\renewcommand{\arraystretch}{1.5}
    f(t)=\left\{
           \begin{array}{ll}
             0, & \hbox{if $t= 0$}; \\
             \displaystyle \frac12, & \hbox{if; $\displaystyle  0< t\leqslant \frac12$}; \\
             1-t, & \hbox{otherwise;}
           \end{array}
         \right.
    \ \ \mbox{and}
    \ \ g(t)=\left\{
                 \begin{array}{ll}
                   0, & \hbox{if $t= 0$}; \\
                   \displaystyle\frac12, & \hbox{if $\displaystyle 0< t\leqslant \frac12$}; \\
                   1-t, & \hbox{otherwise.}
                 \end{array}
               \right.
}\]

The second generator of all copulas in the second row equals $g(t)= t(1-t)$, while the first one equals respectively
\[
{\renewcommand{\arraystretch}{1.5}
    f(t)=\left\{
           \begin{array}{ll}
             2t, & \hbox{if $\displaystyle 0\leqslant t\leqslant \frac13$;} \\
             1-t, & \hbox{otherwise;}
           \end{array}
         \right.
\ \ \ f(t)=\left\{
             \begin{array}{ll}
               \displaystyle\frac12 t, & \hbox{if $\displaystyle0\leqslant t\leqslant\frac23$;} \\
               1-t, & \hbox{otherwise};
             \end{array}
           \right.
}
\]
and
\[
{\renewcommand{\arraystretch}{1.8}
    f(t)=\left\{
           \begin{array}{ll}
             0, & \hbox{if $t=0$;} \\
             \displaystyle\frac13, & \hbox{if $\displaystyle 0<t\leqslant\frac13$;} \\
             t, & \hbox{if $\displaystyle \frac13\leqslant t\leqslant\frac23$;} \\
             1-t, & \hbox{otherwise.}
           \end{array}
         \right.
}
\]\\

Next, we give some regularity properties of the functions $f, g$ and consequently of the reflected maxmin copula.

\begin{proposition}\label{odvod}
    For a given reflected maxmin copula $C_{f,g}$ we have:
     \begin{description}
       \item[(a)] $f$ and $g$ are differentiable a.e. (w.r.\ to Lebesgue measure) on $(0,1)$, and $f'\geqslant-1$ and $g'\geqslant-1$.
       \item[(b)] $C_{f,g}$ is partially differentiable a.e. (w.r.\ to Lebesgue measure) w.r.\ to $u$ and the partial derivative is no greater that $\widehat{g}(v)$.
       \item[(c)] $C_{f,g}$ is partially differentiable a.e. (w.r.\ to Lebesgue measure) w.r.\ to $v$ and the partial derivative is no greater that $\widehat{f}(u)$.
     \end{description}
\end{proposition}

\begin{proof}
\textbf{(a)} Since $\widehat{f}$ is a continuous and nondecreasing function it is firstly differentiable almost everywhere (w.r. to the Lebesgues measure) and secondly $\widehat{f}'(u)\geqslant0$ a.e. So, $f'\geqslant-1$ a.e.\ since $f'(u)=\widehat{f}'(u)-1$ a.e.

\textbf{(b)} By part \textbf{(a)} the function $f$ is differentiable a.e. so that the function $C_{f,g}(u,v)$ is differentiable for every $v$ with respect to $u$ a.e. Consequently, the corresponding partial derivative exists a.e.\ with respect to the Lebesgue measure on $[0,1]^2$. Clearly
\[
    \frac{\partial C_{f,g}}{\partial u}(u,v)=
        \left\{
            \begin{array}{ll}
                0, & \hbox{if $uv<f(u)f(v)$;} \\
                v-f'(u)g(v), & \hbox{otherwise}
            \end{array}
        \right.
\]
for all $v$ and a.e.\ w.r.\ to $u$. Now, using the fact that $f'\geqslant-1$ this expression is no greater than $v+g(v)=\widehat{g}(v)$ for all $v$ and a.e.\ w.r.\ to $u$.

\textbf{(c)} The proof is obtained by exchanging the roles of $u$ and $v$ in the proof of part \textbf{(b)}.
\end{proof}

It is time to present the main results of the section. Some of these are adopted from those presented in \cite{RoLaUbFl} where copulas of the form
\begin{equation}\label{RoLaUbFl}
  C(u,v)=uv+f(u)g(v)
\end{equation}
are considered. However, since we allow $uv-f(u)g(v)$ to be negative and use the truncation in this case, we need to reformulate their results. We include some of the proofs for the benefit of the reader. Let us point out that some extensions of the results of \cite{RoLaUbFl} on copulas of the form \eqref{RoLaUbFl} are also given in \cite{DuJa} and in \cite[Section 1.6]{DuSe}. Also, in the results on absolute continuity of the functions $f$ and $g$, and consequently $C_{f,g}$, the reader will be assumed familiar with definitions and properties of such functions as given in \cite[Section 7.3]{Loja}.

\begin{lemma}\label{abs}
  Suppose that functions $f$ and $g$ satisfy conditions {\rm \textbf{(G1)}--\textbf{(G3)}}. Then:
\begin{description}
  \item[(a)] For any $\varepsilon>0$ both functions are absolutely continuous on the interval $[\varepsilon,1]$.
  \item[(b)] For every pair $\varepsilon,\delta>0$ the product $f(u)g(v)$ is absolutely continuous on $[\varepsilon,1]\times[\delta,1]$.
  \item[(c)] For every pair $\varepsilon,\delta>0$ the difference $uv-f(u)g(v)$ is absolutely continuous on $[\varepsilon,1]\times [\delta,1]$.
\end{description}
\end{lemma}

\begin{proof} \textbf{(a)}
By Lemma \ref{properties G} and the definition following it, $C_{f,g}$ is a copula and has the $2$-increasing property. It follows easily for any points $0\leqslant x_1<x_2\leqslant1$ and $0\leqslant y_1 < y_2 \leqslant 1$ with $f^*(x_1)g^*(y_1)<1$ (implying that $f^*(x_i)g^*(y_j) <1$ for all combinations of indices $i,j\in\{1,2\})$ that
\begin{equation}\label{quotient_product}
  1\geqslant\frac{f(x_2)-f(x_1)}{x_2-x_1}\cdot\frac{g(y_2)-g(y_1)}{y_2-y_1}.
\end{equation}
Here we may choose $x_1\geqslant\varepsilon$ with no loss. As a matter of fact, we may then choose $g$ satisfying conditions \textbf{(G1)}--\textbf{(G3)} in relations \eqref{quotient_product} in an arbitrary manner. In particular, we may choose $g$ non-trivial, so that it contains a point $y_0\in(0,1)$ at which its value is strictly positive, so that it is increasing at some points smaller than $y_0$ and decreasing at some points greater than $y_0$ by the fact that $g(0)=g(1)=0$. So, we may find the points $0\leqslant y_1 < y_2 \leqslant 1$ such that
\[
    \alpha^{-1}:=\frac{g(y_2)-g(y_1)}{y_2-y_1}<0,\ \ \ \mbox{respectively}\ \ \ \beta^{-1}:= \frac{g(y_2)-g(y_1)}{y_2-y_1}>0.
\]
(Observe that we may assume in the first case that $-1\leqslant \alpha^{-1}$ by property \textbf{(G2)}.) 
So, we are developing respectively relation \eqref{quotient_product} into the two relations given below:
\[
    \alpha\leqslant \frac{f(x_2)-f(x_1)}{x_2-x_1} \leqslant \beta,
\]
and $f$ is absolutely continuous by \cite[Lemma 2.1]{RoLaUbFl}. These arguments with the role of $f$ and $g$ reversed yield $g$ being absolutely continuous. \textbf{(b)}  follows by \textbf{(a)} and \textbf{(c)} is an easy consequence of \textbf{(b)}. 
\end{proof}

\begin{corollary}\label{cor:abs}
  An RMM copula $C$ is absolutely continuous on all rectangles $(a,1]\times(b,1]$ and $[0,a)\times[0,b)$ for $(a,b)\in L_0(C)$.
\end{corollary}
\begin{proof}
  This follows immediately by Lemma \ref{abs} and the fact that $C$ is identically $0$ on $\mathcal{Z}(C)$.
\end{proof}

For a function $f$ defined on $[0,1]$ we introduce
\[
A_f=\{t\in[0,1]\,;\,f'(t)\  \mbox{exists}\}.
\]

\begin{lemma}\label{abs_der}
  Suppose that functions $f$ and $g$ satisfy conditions {\rm \textbf{(G1)}--\textbf{(G3)}}. Let $u\in A_f,v\in A_g$ be such that
\begin{description}
  \item[(a)] $f(u)g(v)\leqslant uv$, then $f'(u)g'(v)\leqslant 1$, respectively
  \item[(b)] $f(u)g(v)< uv$, and $f'(u)$ and $g'(v)$ are not both negative, then $f'(u)g'(v)< 1$.
\end{description}
In particular, $f'(u)g'(v)\leqslant 1$ for all  $(u,v)\in \mathcal{S}(C_{f,g}) \cap (A_f\times A_g)$.
\end{lemma}

\begin{proof}
  \textbf{(a)} Condition \textbf{(G2)} clearly implies that ${f^*}'(u) \leqslant0$, implying $f'(u)\leqslant f^*(u)$, and similarly for $g$, so that
\[
    f'(u)g'(v)\leqslant f^*(u)g^*(v)\leqslant1
\]
provided that the derivatives are not both negative. Otherwise the claim follows by Proposition \ref{odvod}\textbf{(a)}. Case \textbf{(b)} goes similarly.
\end{proof}

\begin{theorem}\label{thm:abs}
  Suppose that $C_{f,g}$ is an RMM copula. Then it is absolutely continuous in the interior of the stand of $\mathcal{S}(C_{f,g})$ and in the interior of the zero set $\mathcal{Z}(C_{f,g})$. Therefore, it may be singular only along the boundary $L_0(C_{f,g})$.
\end{theorem}

\begin{proof}
  By Lemma \ref{abs} and Corollary \ref{cor:abs} the functions $f(u)g(v)$ and $C_{f,g}$ are absolutely continuous on all the rectangles in the interior of the stand $\mathcal{S}(C_{f,g})$, and also on all the rectangles in the interior of the zero set $\mathcal{Z}(C_{f,g})$. Therefore, the density function
\begin{equation}\label{density}
  c(u,v)=\left\{
             \begin{array}{ll}
               1-f'(u)g'(v), & \hbox{for $(u,v)\in\mathcal{S}(C_{f,g})$,} \\
               0, & \hbox{for $(u,v)\in\mathcal{Z}(C_{f,g})$,}
             \end{array}
           \right.
\end{equation}
is nonnegative by Lemma \ref{abs_der} and we have
\[
    \int_{a}^{c} \int_{b}^{d}c(u,v)\,du\,dv=C_{f,g}(c,d)-C_{f,g}(c,b)- C_{f,g}(a,d)+C_{f,g}(a,b)
\]
for any rectangle $[a,c]\times[b,d]$ in either the interior of $\mathcal{S}(C_{f,g})$ or of $\mathcal{Z}(C_{f,g})$. Consequently, the copula may be singular only on $L_0(C_{f,g})$.
\end{proof}

\begin{proposition}
  For an RMM copula $C_{f,g}$ the density of its absolutely continuous part is given by \eqref{density}. The mass of its singular part is given by
\[
    1-\iint\limits_\mathcal{S}\left(1-f'(u)g'(v)\right)\,du\,dv= \lambda(\mathcal{Z})+\iint\limits_\mathcal{S}f'(u)g'(v)\,du\,dv.
\]
\end{proposition}

Here we write shorter $\mathcal{S}=\mathcal{S}(C_{f,g})$ and $\mathcal{Z}= \mathcal{Z}(C_{f,g})$. As before $\lambda$ denotes the Lebesque measure.

\begin{proposition}\label{W}
  The only fully singular RMM copula is the Fr\'{e}chet-Hoeffding lower bound $W$.
\end{proposition}

\begin{proof}
  Let $C=C_{f,g}$ be a fully singular RMM copula. So, the mass of its singular part is equal to 1 and by \eqref{density} we have
\begin{equation}\label{one}
  f'(u)g'(v)=1
\end{equation}
for almost all $(u,v)\in\mathcal{S}$. (Here we have implicitly used the fact that by Lemma \ref{abs} functions $f,g$ are absolutely continuous on the interval $(0,1]$ and therefore their derivatives exist almost everywhere.)

Now, if $f'(u)\geqslant0$ or $g'(v)\geqslant0$ for almost all $(u,v)\in \mathcal{S}$, then $f'(u)=g'(v)=0$ (since by \textbf{(G1)} $f(1)=g(1)=0$) contradicting \eqref{one}. So, there exist $u,v\in(0,1)$ such that $f'(u)<0$ and $g'(v)<0$ implying that the set
\[
    \widetilde{\mathcal{S}}=\{(u,v)\in\mathcal{S}\,;\,f'(u)<0\ \mbox{and}\ g'(v)<0\}
\]
is nonempty. We want to show that $\widetilde{\mathcal{S}}$ contains almost all of $\mathcal{S}$. Towards a contradiction, assume this is not the case; then there is a subset of positive Lebesgue measure contained in ${\mathcal{S}}\setminus \widetilde{\mathcal{S}}$. Using \eqref{one} as above we choose a point $(u_1,v_1)$ in this set such that $f'(u_1)>0$ and $g'(v_1)>0$. In addition choose $(u_2,v_2)\in\widetilde{\mathcal{S}}$ and observe that either $(u_1,v_2)\in {\mathcal{S}}$ or $(u_2,v_1)\in {\mathcal{S}}$. Because of $f'(u_1)g'(v_2)<0$ and $f'(u_2)g'(v_1)<0$ this is contradicting \eqref{one}.

Finally, Proposition \ref{odvod} together with \eqref{odvod} tells us that $f'(u)=-1$ and $g'(v)=-1$ for almost all $(u,v)\in \mathcal{S}$ and consequently for almost all $(u,v)\in[0,1]^2$. So,
\[
f(u)=\left\{
       \begin{array}{ll}
         0, & \hbox{$u=0$;} \\
         1-u, & \hbox{$u>0$;}
       \end{array}
     \right.\ \ \ \mbox{and}\ \ \ g(v)=\left\{
       \begin{array}{ll}
         0, & \hbox{$v=0$;} \\
         1-v, & \hbox{$v>0$;}
       \end{array}
     \right.
\]
yielding the desired conclusion.
\end{proof}



The above results on RMM copulas imply the following result for the maxmin copulas.

\begin{corollary}
  Let $C_{\phi,\psi}$ be a maxmin copula. Then the singular component of $C$ is included in the curve $\phi^*(u)=\psi_*(v)$. The only fully singular maxmin copula is the Fr\'{e}chet-Hoeffding upper bound $M$.
\end{corollary}

\begin{proof}
  A short computation reveals that the curve defined by $f^*(u)g^*(v)=1$ corresponds, when translated back to maxmin copulas, to the curve $\phi^*(u)=\psi_*(v)$. Also, when we perform one flip on the Fr\'{e}chet-Hoeffding lower bound $W$, we get the Fr\'{e}chet-Hoeffding upper bound $M$ and the corollary follows by Proposition \ref{W}.
\end{proof}

\textbf{Remark.} It turns out that two pairs of generators $f_1,g_1$ and $f_2,g_2$ that generate the same RMM copula $C_{f_1,g_1}=C_{f_2,g_2}$ (distinct from the product copula) differ only by a multiplicative constant: $f_2=\lambda f_1$ and $g_2=\lambda^{-1} g_1$ for some $\lambda>0$. In particular, the generator $f$ of an SRMM copula $C_f$ is uniquely determined by the copula.

We will omit the proof of these facts not to lengthen this paper too much.

\section{ Statistical aspects of RMM and SRMM copulas }\label{sec:stat}

In this section we need the difference
\begin{equation*}
    \Delta_C(u,v)=\Pi(u,v)-C(u,v)
\end{equation*}
and the corresponding quotient of differences
\[
    Q_C(u_1,v_1;u_2,v_2)=\frac{\Delta_C(u_1,v_1)}{\Delta_C(u_2,v_2)}
\]
for a given RMM copula $C$. Note that difference $\Delta_C$ is always nonnegative and measures the ``dependent'' part of $C$. In particular, it is identically equal to zero if and only if $C\equiv\Pi$ and $C$ models an independence. Of course, the quotient is well-defined only if the denominator is non-zero. 

Given an SRMM copula $C(u,v)$ we want to find the generator $f(u)$ or equivalently $f^*(u)$  in closed form, at least for $u$ for which $f(u)$ is nonzero, to be defined precisely in a moment. (The question of finding one of the two generators of an RMM copula is of the same complexity, as we believe.) Given a continuous nonnegative function $f$ on the interval $[0,1]$ such that \textbf{(G1)}--\textbf{(G3)} hold, the SRMM copula generated by $f$ is defined as
\begin{equation}\label{eq:inverse}
    C_f(u,v)=\max\{0,uv-f(u)f(v)\}.
\end{equation}
Recall the Definition \eqref{eq:new} of $f^*(0)$. For $U\in(0,\min\{1, f^*(0)\}]$ we denote
\[
    {f^*}^{-1}(U):=\inf\{u\in\mathds{R}\,;\,f^*(u)=U\}.
\]
Note that this infimum is actually a minimum by continuity of $f$ and consequently of $f^*$. Now, if $u,v\in(0,1]$ are large enough we get $f^*(u)f^*(v)<1$ by the fact that $f^*$ is nonincreasing and the diagonal $\delta_f(u):=C_f(u,u)$ of the copula defined by \eqref{eq:inverse} is positive for $u$ large enough. Denote $u_{\mathrm{min}}={f^*}^{-1}(\min\{1, f^*(0)\})$ and observe that for all $u,v \geqslant u_{\mathrm{min}}$ we have that $\delta_f(u)$ 
is nonnegative (or equivalently $f^*(v)f^*(u)\leqslant1$) since $f^*$ is nonincreasing. For the same reason it holds for every $x\geqslant1$ that even more $f^*(xv) f^*(u) \leqslant1$. So, the value of the expression
\[
    \frac{f^*(xu)}{f^*(u)}=\frac{1}{x} Q_C(xu,v\,;u,v)
\]
exists as a quotient of two strictly positive numbers and is independent of $v$. Replace in this equation first $u$ by $u_{\mathrm{min}}$ to get
\[
    \frac{f^*(xu_{\mathrm{min}})}{f^*(u_{\mathrm{min}})}=\frac{1}{x} Q_C(xu_{\mathrm{min}},v\,;u_{\mathrm{min}},v).
\]
Then, replace $xu_{\mathrm{min}}$ by an arbitrary $u\geqslant u_{\mathrm{min}}$ to get the formula (after recalling from the definition of $u_{\mathrm{min}}$ that $f^*(u_{\mathrm{min}})=1$ and consequently $f(u_{\mathrm{min}})=u_{\mathrm{min}}$)
\begin{equation}\label{generating}
    f(u) =u_{\mathrm{min}} Q_C(u,v\,;u_{\mathrm{min}},v).
\end{equation}
Here, $v$ is arbitrary but chosen so that both differences in the numerator and the denominator of this quotient are positive. Observe that the quotient is independent of $v$.

\begin{theorem}\label{thm:closedform}
  An RMM copula $C$, generated by generators of two SRMM copulas $C_1$ and $C_2$, is given in closed form by
\[
    C(u,v)=\max\left\{0,uv-u_{\mathrm{min}} Q_{C_1}(u,w_1\,;u_{\mathrm{min}},w_1)
    v_{\mathrm{min}} Q_{C_2}(v,w_2\,;v_{\mathrm{min}},w_2)
    \right\}.
\]
\end{theorem}

  \prf
In order to get this formula we use formula \eqref{generating} twice. In the first factor we replace $v$ by $w_1$, while in the second one we replace $u_{\mathrm{min}}$ by $v_{\mathrm{min}}$, $u$ by $v$, and $v$ by $w_2$.
  \qed\\

Observe that this formula is independent of the choice of $w_1$ and $w_2$ provided that in both quotients the numerators and denominators are nonzero (and, of course we can always choose them so.)

\begin{theorem}
    Suppose that $C$ is an RMM copula, and $F$ and $G$ are univariate d.f.'s. Furthermore, let $H(x,y)=C(F(x),G(y))$ be a bivariate d.f. Then there exist three independent random variables $Z_1$, $Z_2$ and $Z_3$ such that $H$ is the joint distribution function of the random pair $(X_1,Y_2)$, where the d.f.\ of $Y_2$ is the survival d.f.\ corresponding to $X_2$, and
    \[
        X_1=\max\{Z_1,Z_3\}\ \ \mbox{and}\ \ X_2=\min\{Z_2,Z_3\}.
    \]
\end{theorem}

This follows directly from our definition and \cite[Lemma 7]{OmRu}. We now want to combine all the results of this section together with Formula \eqref{generating} for inverted marginal d.f. Observe in particular that when $g(v)=f(u)$ with $v=1-u$ the point $v_{\mathrm{min}}$ translates into $u_{\mathrm{max}}$.

\begin{theorem}
    Suppose that $C$ is a maxmiin copula, and $F$ and $G$ are univariate d.f.'s. Furthermore, let $H(x,y)=C(F(x),G(y))$ be a bivariate d.f. Then  there exist three independent random variables $Z_1$, $Z_2$ and $Z_3$ such that $H$ is the joint distribution function of random pair $(X_1,X_2)$, where
    \[
        X_1=\max\{Z_1,Z_3\}\ \ \mbox{and}\ \ X_2=\min\{Z_2,Z_3\}.
    \]
    Moreover, there exist two SRMM copulas $C_1$ and $C_2$ such that
    \begin{eqnarray*}
    C(u,v) &=& u- \max\{0,u(1-v) \\
    &-& u_{\mathrm{min}} Q_{C_1}(u,w_1\,;u_{\mathrm{min}},w_1)
        (1-v_{\mathrm{max}}) Q_{C_2}(1-v,w_2\,;1-v_{\mathrm{max}},w_2)
        \}.
    \end{eqnarray*}
\end{theorem}

\section{ Diagonals of symmetric reflected maxmin copulas }\label{sec:diag}

One wants to determine which diagonal functions $\delta$ that satisfy $\delta(t)\leqslant t^2$ on $[0,1]$ are possible diagonal sections of RMM copulas. We say that a function $\delta:[0,1]\rightarrow[0,1]$ is a \emph{diagonal function} if it satisfies the following conditions:
\begin{description}
  \item[(D1)] $\delta(0)=0,\delta(1)=1$,
  \item[(D2)] $\delta(t)\leqslant t$ for all $t\in[0,1]$,
  \item[(D3)] $\delta$ is nondecreasing,
  \item[(D4)] $\delta$ is 2-Lipshitz: $|\delta(s)-\delta(t)|\leqslant 2|s-t|$ for all $s,t\in[0,1]$.
\end{description}
A function $\delta:[0,1]\rightarrow[0,1]$ is a diagonal function if and only if it is the diagonal section of a copula \cite[pp.\ 84-85]{Nels}. We denote the set of all diagonal functions by $\mathcal{D}$.

\begin{proposition}\label{prop:delta1}
Suppose $\delta$ is a diagonal section of an RMM copula $C_{f,g}$. Then $\displaystyle \delta^\#(t) =\frac{\delta(t)}{t^2}:(0,1]\rightarrow[0,1]$ is nondecreasing.
\end{proposition}

  \prf
We have that
\[
    \delta(t)=t^2\max\{0,1-f^*(t)g^*(t)\}.
\]
By Lemma \ref{properties G} (property \textbf{(G3)}, and combined with \textbf{(G1)}) the functions $f^*(t)$ and $g^*(t)$ are nonincreasing and nonnegative, so that their product $f^*g^*$ is also nonincreasing. Hence, the function $1-f^*(t)g^*(t)$ is nondecreasing and proposition follows.
  \qed\\

We can say somewhat more about the diagonal sections in case of the symmetric RMM's.

\begin{proposition}\label{prop:delta2}
Suppose $\delta$ is a diagonal section of an SRMM copula $C_{f}$. Then $\widehat{\delta}(t) =t+ \sqrt{t^2-\delta(t)} :[0,1]\rightarrow[0,\infty)$ is nondecreasing.
\end{proposition}

  \prf
Since $\delta(t)=\max\{0,t^2-f(t)^2\}$ we have $\widehat{\delta}(t) =t+ \min\{t,f(t)\}$, i.e.
\begin{equation}\label{delta}
    \widehat{\delta}(t) =\left\{
                           \begin{array}{ll}
                             2t, & \hbox{if $t\leqslant f(t)$;} \\
                             t+f(t), & \hbox{if $t\geqslant f(t)$.}
                           \end{array}
                         \right.
\end{equation}
By Lemma \ref{properties G} (property \textbf{(G2)}) it follows that the function $t+f(t)$ is nondecreasing. Then, $\widehat{\delta}$ is nondecreasing by an easy observation based on Equation \eqref{delta}.
  \qed\\

We denote by $\widehat{\mathcal{D}}$ the set of all diagonal sections that satisfy conditions of Propositions \ref{prop:delta1} and \ref{prop:delta2}, i.e.
\[
    \widehat{\mathcal{D}}=\{\delta\in\mathcal{D};\, \mbox{s.t.}\ \delta^\#, \widehat{\delta}\ \mbox{are nondecreasing} \}.
\]
Observe that the functions $\delta^\#$, and $\widehat{\delta}$ are related via Equations
\begin{equation}\label{delta hat}
    \widehat{\delta}(t)=t\left(1+\sqrt{1-\delta^\#(t)}\right)
\end{equation}
and
\begin{equation}\label{delta sharp}
    \delta^\#(t)=1-\left(\frac{\widehat{\delta}(t)}{t}-1\right)^2,
\end{equation}
where we used the standing assumption $\delta(t)\leqslant t^2$ for all $t\in[0,1]$ that holds for diagonal sections of reflected maxmin copulas. From relations \eqref{delta hat} and \eqref{delta sharp} it follows directly that $\delta^\#$ is nondecreasing if and only if $\displaystyle \frac{\widehat{\delta}(t)}{t}$ is nonincreasing on $(0,1]$. Using this fact we get an equivalent definition of $\widehat{\mathcal{D}}$
\[
    \widehat{\mathcal{D}}=\{\delta\in\mathcal{D};\, \mbox{s.t.}\ \widehat{\delta}(t)\ \mbox{is nondecreasing and}\ \frac{\widehat{\delta}(t)}{t}\ \mbox{is nonincreasing} \}.
\]

\begin{theorem}\label{thm:main}
  A function $\delta$ is a diagonal section of an SRMM copula if and only if $\delta\in \widehat{\mathcal{D}}$. Moreover, if $f(t)=\sqrt{t^2-\delta(t)}$ for $t\in[0,1]$, where $\delta\in \widehat{\mathcal{D}}$, then $\delta$ is the diagonal section of the SRMM copula $C_f$.
\end{theorem}

  \prf
Suppose first that $\delta(t)$ is the diagonal section of the SRMM copula $C_f$. Then $\delta\in \widehat{\mathcal{D}}$ by Propositions \ref{prop:delta1} and \ref{prop:delta2}.

Conversely, if $\delta\in \widehat{\mathcal{D}}$ then $\delta^\#$, and $\widehat{\delta}$ are nondecreasing. Define $f(t)= \sqrt{t^2-\delta(t)}, f:[0,1]\rightarrow[0,1]$ and observe first that it is well-defined. Indeed, $t^2-\delta(t)\geqslant0$, $\displaystyle \delta^\#(t)=\frac{\delta(t)}{t^2}$ is nondecreasing on $(0,1]$, and $\delta^\#(1)=1$. Next observe that
\[
    \widehat{f}(t)=t+f(t)=t+\sqrt{t^2-\delta(t)}=\widehat{\delta}(t),
\]
and
\[
    f^*(t)=\frac{f(t)}{t}=\sqrt{1-\delta^\#(t)}.
\]
Since $\delta^\#$ and $\widehat{\delta}$ are nondecreasing it follows easily that $\widehat{f}$ is nondecreasing and $f^*(t)$ is nonincreasing. Recall that $f(0)=f(1)=f^*(1)=0$ to conclude that all the properties \textbf{(G1)}--\textbf{(G3)} hold and $f$ is a generator of an SRMM copula with the diagonal $\delta$.
  \qed\\

\begin{example}
An SRMM copula is not uniquely determined by its diagonal section in general.
\end{example}

  \prf
The diagonal of the Fr\'{e}chet-Hoeffding lower bound $W$ equals
\[
    \delta_W(t)=\max\{0,2t-1\}.
\]
It belongs to $\widehat{\mathcal{D}}$. We already know that $W=C_g$ for
\[
    g(t)=\left\{
           \begin{array}{ll}
             0, & \hbox{if $t=0$;} \\
             1-t, & \hbox{if $0<t\leqslant1$.}
           \end{array}
         \right.
\]
Observe also that $g$ is not equal to the function $f$ corresponding to $\delta_W$ by Theorem \ref{thm:main}. This is equal to $f(t)=\min\{t,1-t\}$ and the corresponding copula is equal to
\[{\renewcommand{\arraystretch}{1.8}
    C_f(u,v)=\left\{
               \begin{array}{ll}
                 0, & \hbox{$\displaystyle 0\leqslant u,v\leqslant\frac{1}{2}$;} \\
                 u(2v-1), & \hbox{$\displaystyle0\leqslant u\leqslant\frac{1}{2}\leqslant v\leqslant1$;} \\
                 v(2u-1), & \hbox{$\displaystyle0\leqslant v\leqslant\frac{1}{2}\leqslant u\leqslant1$;} \\
                 u+v-1, & \hbox{$\displaystyle\frac{1}{2}\leqslant u,v\leqslant1$.}
               \end{array}
             \right.
}\]
Actually, for any function $h:[0,1]\rightarrow[0,1]$ such that $f(t)\leqslant h(t)\leqslant g(t)$ for all $t\in[0,1]$ and such that $\widehat{h}$ and $h^*$ are nondecreasing and nonincreasing, respectively, we have that the diagonal of $C_h$ is equal to $\delta_W$. Observe that the copula $C_f$ is equal to the reflected ordinal sum of two copies of $\Pi$ with respect to the partition $\displaystyle\{0,\frac12,1\}$ (see \cite[Section 3.8]{DuSe}).
  \qed\\

\begin{corollary}
If $\delta\in \widehat{\mathcal{D}}$ is such that $\delta(t)>0$ for all $t\in(0,1]$ then $C_f$ for $f(t)=\sqrt{t^2-\delta(t)}$ is the unique SRMM copula with the diagonal section equal to $\delta$.
\end{corollary}

  \prf
By Theorem \ref{thm:main} we know that $f$ is such that the diagonal section of $C_f$ is equal to $\delta$. Suppose that $C_g(u,v)=\max\{0, uv-g(u)g(v)\}$ is an SRMM copula such that its diagonal is equal to $\delta$. Then $\delta(t)=\max\{0,t^2-g(t)^2\}$ and since $\delta(t)>0$, we have that $t^2-g(t)^2>0$ for all $t>0$. But then $g(t)=\sqrt{t^2-\delta(t)}=f(t)$ for all $t\in[0,1]$.
  \qed\\

\begin{example}
There exists a diagonal section $\delta$ of a copula satisfying $\delta(t)\leqslant t^2$ for all $t$ such that $\delta^\#$ is nondecreasing and $\widehat{\delta}$ is not.
\end{example}

  \prf
Let $\delta$ be given by
\[{\renewcommand{\arraystretch}{2.2}
    \delta(t)=\left\{
                \begin{array}{ll}
                  0, & \hbox{$\displaystyle t\in\left[0,\frac{1}{4}\right]$;} \\
                  \displaystyle 2t-\frac{1}{2}, & \hbox{$\displaystyle t\in\left[\frac{1}{4},1-\frac{\sqrt{2}}{2}\right]$;} \\
                  t^2, & \hbox{$\displaystyle t\in\left[1-\frac{\sqrt{2}}{2},1\right]$.}
                \end{array}
              \right.
}\]
It is easy to verify that $\delta$ is a continuous function that satisfies conditions \textbf{(D1)}--\textbf{(D4)} together with $\delta(t)\leqslant t^2$ for all $t\in[0,1]$. Then
\[{\renewcommand{\arraystretch}{2.2}
    \delta^\#(t)=\left\{
                \begin{array}{ll}
                  0, & \hbox{$\displaystyle t\in\left[0,\frac{1}{4}\right]$;} \\
                  \displaystyle \frac{4t-1}{2t^2}, & \hbox{$\displaystyle t\in\left[\left(\frac{1}{4}\right)^+,1-\frac{\sqrt{2}}{2}\right]$;} \\
                  1, & \hbox{$\displaystyle t\in\left[1-\frac{\sqrt{2}}{2},1\right]$.}
                \end{array}
              \right.
}\]
is nondecreasing while
\[{\renewcommand{\arraystretch}{2.2}
    \widehat{\delta}(t)=\left\{
                \begin{array}{ll}
                  2t, & \hbox{$\displaystyle t\in\left[0,\frac{1}{4}\right]$;} \\
                  \displaystyle t+\sqrt{t^2-2t+\frac{1}{2}}, & \hbox{$\displaystyle t\in\left[\frac{1}{4},1-\frac{\sqrt{2}}{2}\right]$;} \\
                  t, & \hbox{$\displaystyle t\in\left[1-\frac{\sqrt{2}}{2},1\right]$.}
                \end{array}
              \right.
}\]
is not.
  \qed\\

For any diagonal section $\delta$ denote by $a_\delta$ the largest $t\in[0,1]$ such that $\delta(a_\delta)=0$. This means that $\delta(t)=0$ for $t\in[0,a_\delta]$ and $\delta(t)>0$ for $t\in(a_\delta,1]$. Since $W$ is the lower bound for all copulas it follows that $a_\delta\in\left[0,\frac{1}{2}\right]$. In the following theorem we give the lower and the upper bound for SRMM copulas with a given diagonal section $\delta$.

\begin{theorem}\label{thm:bounds}
Suppose that $\delta\in \widehat{\mathcal{D}}$ is such that $a_\delta>0$ and write
\[{\renewcommand{\arraystretch}{1.2}
    h(t)=\left\{
           \begin{array}{ll}
             0, & \hbox{$t=0$;}\\
             2a_\delta-t, & \hbox{$t\in(0,a_\delta]$;} \\
             \sqrt{t^2-\delta(t)}, & \hbox{$t\in(a_\delta,1]$;}
           \end{array}
         \right.\ \ \mbox{and}\ \ \ f(t)=\sqrt{t^2-\delta(t)}.
}\]
Then $C_h$ and $C_f$ are SRMM copulas with diagonal section $\delta$, and if $C_g$ is any SRMM copula with diagonal section $\delta$, then
\[
    C_h(u,v)\leqslant C_g(u,v)\leqslant C_f(u,v)\ \ \mbox{for all}\ \ u,v\in[0,1].
\]
\end{theorem}

  \prf
The fact that $C_f$ is an SRMM copula follows by Theorem \ref{thm:main}. To see that the same is true for $C_h$ we have to modify slightly the proof of that theorem. We first observe that $h(0)=h(1)=h^*(1)=0$ and then introduce
\[{\renewcommand{\arraystretch}{1.2}
    \widehat{h}(t)= t+h(t)=\left\{
           \begin{array}{ll}
             t, & \hbox{$t=0$;}\\
             2a_\delta, & \hbox{$t\in(0,t_\delta]$;} \\
             t+\sqrt{t^2-\delta(t)}, & \hbox{$t\in(t_\delta,1]$.}
           \end{array}
         \right.
}\]
As in the proof of Theorem \ref{thm:main} we see that the function $\widehat{h}(t)$ is nondecreasing on $[a_\delta,1]$ and clearly it is also nondecreasing on $[0,a_\delta]$. Since $\delta(a_\delta)=0$, and so $a_\delta+ \sqrt{a_\delta^2-\delta(a_\delta)}=2a_\delta$, it is nondecreasing on $[0,1]$.

Next, we consider the function
\[{\renewcommand{\arraystretch}{1.5}
    h^*(t)= \left\{
           \begin{array}{ll}
             \displaystyle 2\frac{a_\delta}{t}-1, & \hbox{$t\in(0,t_\delta]$;} \\
             \sqrt{1-\delta^\#(t)}, & \hbox{$t\in(t_\delta,1]$}
           \end{array}
         \right.
}\]
for $t\in(0,1]$. This function is clearly nonincreasing on $(0,a_\delta]$ and we can see, similarly to the proof of Theorem \ref{thm:main}, that it is nonincreasing on $(a_\delta,1]$. Since $\displaystyle \sqrt{1-\delta^\#(a_\delta)} =1=2\frac{a_\delta}{t}-1$ it is nonincreasing on $[0,1]$. So, $h$ satisfies conditions \textbf{(G1)}--\textbf{(G3)} and it is a generator of an SRMM copula.

Now, suppose $C_g$ is an SRMM copula with diagonal section $\delta$. So, $C_g(t,t)=\max\{0,t^2- g(t)^2\} =\delta(t)$. On $(0,a_\delta]$ we have $\delta(t)=0$ so that
\begin{equation}\label{eq:bounds one}
    g(t)\geqslant t\ \ \mbox{on}\ \ [0,a_\delta].
\end{equation}
For $t\in[a_\delta,1]$ we have $g(t)=\sqrt{t^2-\delta(t)}$. The fact that $\widehat{g}(t)$ is nondecreasing on $[0,1]$ implies $t+g(t)\leqslant a_\delta+g(a_\delta)=2a_\delta$ on $[0,a_\delta]$ so that
\begin{equation}\label{eq:bounds two}
    g(t)\leqslant2a_\delta-t\ \ \mbox{on}\ \ [0,a_\delta].
\end{equation}
Combining \eqref{eq:bounds one} and \eqref{eq:bounds two} we obtain
\[
    f(t)\leqslant g(t)\leqslant h(t)\ \ \mbox{on}\ \ [0,a_\delta]\ \ \mbox{and}\ \
    f(t) =   g(t) =   h(t)\ \ \mbox{on}\ \ [a_\delta,1].
\]
The desired inequality then follows easily.
  \qed\\

This theorem implies, in particular, that the SRMM copula is uniquely determined by the diagonal section $\delta$ for $u,v\in[a_\delta,1]$.

We will now show by a counterexample that the statement of Proposition \ref{prop:delta2} fails if SRMM coupula is replaced by a general RMM copula.

\begin{example}
There exists an RMM copula such that for its diagonal $\delta(t)$ function $\widehat{\delta}(t)=t+ \sqrt{t^2-\delta(t)}$ is not nondecreasing.
\end{example}

 \prf
It is not hard to see that functions
\[{\renewcommand{\arraystretch}{1.2}
    f(t)=\left\{
           \begin{array}{ll}
             t, & \hbox{$\displaystyle 0\leqslant t\leqslant \frac{1}{2}$;} \\
             1-t, & \hbox{$\displaystyle \frac{1}{2}\leqslant t\leqslant 1$;}
           \end{array}
         \right.\ \ \
         \hbox{and}\ \ \ g(t)=\left\{
                                \begin{array}{ll}
                                  \displaystyle\frac{1}{2}-t, & \hbox{$\displaystyle 0\leqslant t\leqslant \frac{1}{2}$;} \\
                                  0, & \hbox{$\displaystyle \frac{1}{2}\leqslant t\leqslant 1$;}
                                \end{array}
                              \right.
}\]
satisfy conditions \textbf{(G1)}--\textbf{(G3)}. Thus, they are generating functions of the RMM copula $C_{f,g}(u,v)$. The diagonal section of this copula equals
\[{\renewcommand{\arraystretch}{2}
    \delta(t)=C_{f,g}(t,t)=\max\{0,t^2-f(t)g(t)\}=\left\{
                                                    \begin{array}{ll}
                                                      0, & \hbox{$\displaystyle 0\leqslant t\leqslant \frac{1}{4}$;} \\
                                                      \displaystyle 2t^2-\frac{1}{2}t, & \hbox{$\displaystyle \frac{1}{4}\leqslant t\leqslant \frac{1}{2}$;} \\
                                                      t^2, & \hbox{$\displaystyle \frac{1}{2}\leqslant t\leqslant 1$.}
                                                    \end{array}
                                                  \right.
}\]
Then $\widehat{\delta}(t)$ is equal to
\[
    {\renewcommand{\arraystretch}{1.8}
    \widehat{\delta}(t)=\left\{
                          \begin{array}{ll}
                            2t, & \hbox{$\displaystyle 0\leqslant t\leqslant \frac{1}{4}$;} \\
                            t+\sqrt{\frac{1}{2}t-t^2}, & \hbox{$\displaystyle \frac{1}{4}\leqslant t\leqslant \frac{1}{2}$;} \\
                            t, & \hbox{$\displaystyle \frac{1}{2}\leqslant t\leqslant 1$.}
                          \end{array}
                        \right.}
\]
Since
\[
    \widehat{\delta}\left(\frac{3}{8}\right)=\frac{3+\sqrt{3}}{8}>\frac{4}{8}=\widehat{\delta}\left(\frac{1}{2}\right)
\]
it follows that $\widehat{\delta}$ is not nondecreasing.
 \qed\\[5mm]

\textbf{Acknowledgements.}
The authors are thankful to Fabrizio Durante for some useful comments to a previous version of this paper. We are also grateful to an anonymous referee for helping us improve some parts of the paper especially in Section 3. We want to thank to our colleague Bla\v{z} Moj\v{s}kerc who helped us with the figures.

\end{document}